%% file: main_quantum.tex
\documentclass[a4paper,onecolumn,11pt,accepted=2026-06-23]{quantumarticle}
\pdfoutput=1

\input{_macros_quantum.tex}
\usepackage{quantikz}
\usepackage{braket}

\begin{document}

\title{Quantum Alternating Direction Method of Multipliers for Semidefinite Programming}
\makeatletter
\def\@printtitle{%
  \href{https://quantum-journal.org/?s=Quantum\%20Alternating\%20Direction\%20Method\%20of\%20Multipliers\%20for\%20Semidefinite\%20Programming\&reason=title-click}{%
    \color{quantumviolet}{%
      \parbox{\linewidth}{\raggedright\huge
        Quantum Alternating Direction Method of\\[0.2ex]
        Multipliers for Semidefinite Programming}%
    }%
  }%
}
\makeatother

\author{Hantao Nie}
\affiliation{School of Mathematical Sciences, Peking University, Beijing, China}
\email{nht@pku.edu.cn}
\author{Dong An}
\affiliation{Beijing International Center for Mathematical Research, Peking University, Beijing, China}
\email{dongan@pku.edu.cn}
\author{Zaiwen Wen}
\affiliation{Beijing International Center for Mathematical Research, Peking University, Beijing, China}
\email{wenzw@pku.edu.cn}

\maketitle

\begin{abstract}

Semidefinite programming (SDP) is a fundamental convex optimization problem with wide-ranging applications. However, solving large-scale instances remains computationally challenging due to the high cost of solving linear systems and performing eigenvalue decompositions. In this paper, we present a quantum alternating direction method of multipliers (QADMM) for SDPs, building on recent advances in quantum computing. An inexact ADMM framework is developed, which tolerates errors in the iterates arising from block-encoding approximation and quantum measurement. Within this robust scheme, we design a polynomial proximal operator to address the semidefinite conic constraints and apply the quantum singular value transformation to accelerate the most costly projection updates. We prove that the scheme converges to an $\epsilon$-optimal solution of the SDP problem under the strong duality assumption. A detailed complexity analysis shows that the QADMM algorithm achieves favorable scaling with respect to dimension compared to the classical ADMM algorithm and quantum interior point methods, highlighting its potential for solving large-scale SDPs.

\end{abstract}

\section{Introduction}
Semidefinite programming (SDP) is a pivotal class of convex optimization problems that generalizes linear programming by extending the optimization variable from vectors to matrices and requiring a positive semidefinite constraint. 
We denote the real matrix inner product $\langle \cdot, \cdot \rangle$ as $\langle A, B\rangle = \mathrm{Tr}(A^{\top}B)$. Given $m$ symmetric real matrices $A^{(i)} \in \mathbb{S}^n := \{ X \in \mathbb{R}^{n \times n}| X^\top = X \}, i=1, \ldots, m$, define the linear map $\mathcal{A}: \mathbb{R}^{n \times n} \rightarrow \mathbb{R}^m$ and its adjoint operator $\mathcal{A}^*: \mathbb{R}^m \rightarrow \mathbb{S}^n$ as
\begin{equation}
\label{eq: linear map A}
\begin{aligned}
\mathcal{A}(X)&:=\bigl(\langle A^{(1)},X\rangle,\ldots,
\langle A^{(m)},X\rangle\bigr)^{\top},\\
\mathcal{A}^*(y)&:=\sum_{i=1}^m y^i A^{(i)} ,
\end{aligned}
\end{equation}
respectively. Further given a symmetric real matrix $C \in \mathbb{S}^n$, and a vector $b \in \mathbb{R}^m$, the SDP and its dual problem are formulated as
\begin{subequations}
    \begin{align}
\min _{X \in \mathbb{S}^n} \quad \langle C, X\rangle \quad &\text { s.t. } \quad \mathcal{A}(X)=b, \quad X \succeq 0,     \label{eq: primal SDP}
\\
\min _{y \in \mathbb{R}^m, S \in \mathbb{S}^n}-b^{\top} y \quad &\text { s.t. } \quad \mathcal{A}^*(y)+S=C, \quad S \succeq 0 , \label{eq: dual SDP}
    \end{align}
\end{subequations}
respectively.
For normalization purposes, we assume $\| C \|_{\mathrm{F}} , \| A^{(i)}\|_{\mathrm{F}} \leq 1, i=1, \ldots, m$. Following~\cite{van2018improvements}, the solutions are assumed to be bounded, i.e., there exist explicit known upper bounds $R_X, R_y, R_S$ such that a triplet of primal-dual optimal solutions $(X^*, y^*, S^*)$ satisfies
$\|X^*\|_{\mathrm{F}} \leq R_X$, $\|y^*\|_1 \leq R_y$, and $\|S^*\|_{\mathrm{F}} \leq R_S$. Under this assumption, the strong duality holds for~\eqref{eq: primal SDP} and~\eqref{eq: dual SDP}, i.e., the optimal values of the primal and dual problems are equal.

The significance of solving SDPs lies in their ability to provide a versatile framework for formulating critical problems across various domains, including control theory~\cite{boyd1994linear, wolkowicz2012handbook, fares2002robust, majumdar2020recent}, statistics~\cite{banks2021local, bertsimas1998semidefinite, vandenberghe1999applications}, machine learning~\cite{lanckriet2004learning, de2006semi, fazlyab2021introduction, zhang2006ensemble}, finance~\cite{gepp2020financial, leibfritz2009successive} and quantum information science~\cite{fawzi2019semidefinite, berta2022semidefinite, mironowicz2024semi, skrzypczyk2023semidefinite, wang2018semidefinite, eldar2003semidefinite}.
Moreover, SDP provides convex relaxations for several classical combinatorial optimization problems, such as max-cut~\cite{goemans1995improved, waldspurger2015phase} and quadratic assignment problems~\cite{de2012improved, zhao1998semidefinite}.
 
Classical methods for solving SDP encompass
cutting-plane methods~\cite{lee2015faster}, matrix multiplicative weights updates (MWU)~\cite{arora2006multiplicative},
interior point methods (IPMs) and operator splitting-based methods.
The state-of-the-art cutting-plane method~\cite{lee2015faster} solves the SDP in \(
\widetilde{\mathcal{O}} \!\left(m\!\left(m^{2}+n^{\omega}+s_A\right)\right)\) time, where \(\omega<2.376\) is the matrix multiplication exponent
and \(s_A\) is the sum of the number of non-zero entries in the coefficient matrix \(A^{(i)}, i=1,\ldots,m\). The MWU framework applies mirror descent method to SDP under the relative entropy divergence and attains \(\mathcal{O}\left(mn^{2}(\frac{R_{\mathrm{Tr}X} R_y}{\epsilon_{\mathrm{abs}}})^4\right)\) complexity, where \(R_{\mathrm{Tr}X}\) and \(R_y\) are the upper bounds of $\mathrm{Tr}(X^*)$ and $\|y^*\|_1$, respectively, and \(\epsilon_{\mathrm{abs}}\) denotes absolute accuracy in the objective value and the relevant feasibility residual.
IPM is a well-established class of algorithms based on the central path and Newton's method, which is also widely used in practice, particularly for medium-scale SDPs.
The enhanced IPM~\cite{deng2024enhanced}
requires only \({O}(\sqrt{n}\log(\frac{1}{\epsilon_{\mathrm{gap}}}))\) iterations to reach \(\epsilon\)-optimality, where \(\epsilon_{\mathrm{gap}}\) is the normalized duality gap, i.e., $\epsilon_{\mathrm{gap}}$-optimality is achieved $\frac{\langle X, S\rangle}{n} \leq \epsilon_{\mathrm{gap}}$ while
each iteration costs \({O}( s_A + m^\omega+n^\omega)\). 
By applying 
alternating direction method of multipliers (ADMM)
~\cite{boyd2011distributed}
or primal-dual hybrid gradient (PDHG)~\cite{chambolle2011first} schemes to SDP, operator splitting-based methods such as SDPAD~\cite{wen2010alternating}, ABIP~\cite{lin2022alternating} and SCS~\cite{o2016conic} have gained popularity for solving large-scale instances.
As a representative example,
ABIP~\cite{deng2024enhanced} is proved to converge in a number of iterations as
$
    \mathcal{O}\left(\frac{\kappa_A^2\|\mathbf{Q}\|^2}{\sqrt{\epsilon_{\mathrm{abs}}}} \log \left(\frac{1}{\epsilon_{\mathrm{abs}}}\right)\right)
$,
where $\kappa_A$ is the condition number of the primal-dual pair of SDPs, $\mathbf{Q}$ is the coefficient matrix of the homogeneous formulation, $\|\mathbf{Q}\|^2 = \mathcal{O}(\sum_{i=1}^m \|A_i\|^2 + \|b\|^2 + \|C\|^2)$.
Each iteration in ABIP costs $\mathcal{O}(n^4)$ operations and requires one eigenvalue decomposition.

To overcome the high computational cost of solving large-scale SDPs, quantum algorithms have been proposed to leverage the power of quantum computing, mainly including quantum matrix multiplicative weights updates (QMWU)~\cite{brandao2017quantum}, quantum interior point methods (QIPM)~\cite{augustino2023quantum, mohammadisiahroudi2025quantum}, where $s$ is a sparsity parameter (maximum number of non-zero entries per row, $s \leq n$).
QMWU realizes a quadratic improvement over classic MWU, running in \(\widetilde{\mathcal{O}}\left(\sqrt{mn}\,s^{2}\right)\) queries which matches \(\Omega(\sqrt{m}+\sqrt{n})\) lower bounds on dimension~\cite{brandao2017quantum}.  The subsequent refinements reduce overheads in sparsity, error and boundedness parameters~\cite{van2018improvements, van2017quantum} and achieve complexity of $\widetilde{\mathcal{O}} \left((\sqrt{m}+\sqrt{n} \frac{R r}{\epsilon_{\mathrm{abs}}}) s (\frac{R_{\mathrm{Tr}X} R_y}{\epsilon_{\mathrm{abs}}})^4\right)$, where $R_{\mathrm{Tr}X}$ and $R_y$ have the same meaning as in MWU.
QIPM~\cite{augustino2023quantum} improves the complexity of classical interior point methods by leveraging quantum linear system solver for the Newton system, achieving a complexity of \(\widetilde{\mathcal{O}}\left(n^{3.5}\frac{\kappa_{newt}^2}{\epsilon_{\mathrm{gap}}}\right)\) quantum accesses and \(\widetilde{\mathcal{O}}\left(n^{4.5}\right)\) arithmetic operations for solving the primal-dual pair of SDPs, where \(\kappa_{newt}\) is the condition number of the Newton system.
An iterative-refinement framework for QIPM~\cite{mohammadisiahroudi2025quantum} further reduces the complexity of QIPM to \(\widetilde{\mathcal{O}}\left(n^{3.5}\kappa_0^2\right)\) quantum accesses and \(\widetilde{\mathcal{O}}\left(n^{4.5}\right)\) arithmetic operations, where \(\kappa_0\) is the condition number of the initial Newton system.
A detailed and comprehensive review of classical and quantum algorithms is presented in Appendix~\ref{app:sec: literature review}.

\subsection{Our Contributions}

In this paper, we propose a quantum alternating direction method of multipliers (QADMM) approach  that combines first-order algorithms with quantum speed-ups for eigenvalue transformation.
Our main contributions are listed as follows.

1. 
Inexact ADMM framework for SDP tolerant to quantum errors. In addition to allowing errors in subproblem solves, as is standard in the inexact ADMM literature, we also allow errors in the dual variable update step.
This design enables the algorithm to tolerate the inherent errors introduced by quantum subroutines, such as block-encodings and tomography. It is proved that the averaged iterate of this inexact ADMM scheme converges to an $\epsilon$-optimal solution with  convergence rate $\mathcal{O}(\frac{1}{\epsilon})$. This inexact ADMM framework for SDP provides a robust theoretical bridge between classical ADMM theory and practical quantum computations, ensuring that quantum acceleration can be harnessed without sacrificing convergence guarantees.

2. 
QADMM algorithm for SDP with polynomial proximal mapping. 
The core idea of our QADMM algorithm is to replace the expensive eigenvalue decomposition steps in the classical ADMM with quantum subroutines.
Unlike a traditional log-barrier or direct projection onto the positive semidefinite cone which would require costly eigenvalue decomposition at each step, we consider crafting an implicit barrier function so that its proximal operator reduces to a polynomial eigenvalue transformation. 
By implementing the subproblem update step utilizing quantum singular value transformation (QSVT) techniques under the block-encoding and QRAM data-access assumptions, we obtain a favorable dimension dependence for the overall QADMM complexity.

3. 
Complexity analysis and comparison with existing methods. We present a rigorous complexity analysis for the proposed QADMM algorithm, demonstrating its favorable theoretical performance compared to both classical and quantum state-of-the-art SDP solvers. In particular, we prove that our algorithm converges to an $\epsilon_{\mathrm{abs}}$-accurate optimal solution, where $\epsilon_{\mathrm{abs}}$ denotes absolute accuracy in the objective value and the relevant feasibility residual,
with at most
$
\widetilde{\mathcal{O}}\left(\left(m \kappa_A^2( 1 + R_y)^2 + n^2 ({\kappa_A^2}( 1 + R_y) + R_X )\right)\frac{(R_X+R_S)^3}{\epsilon_{\mathrm{abs}}^3}\right)
$
 quantum gate complexity, and
$O\left((s_A + n^2)\frac{R_X^2 + R_S^2}{\epsilon_{\mathrm{abs}}}\right)$ classical arithmetic operations, 
where $R_X$ and $R_S$ are the upper bounds of primal and dual solutions $X^*$ and $S^*$, respectively, 
$\kappa_A^2$ is the condition number of $\mathcal{A}\mathcal{A}^*$, and $s_A$ is the total number of non-zero entries in $A^{(i)}, i=1,\ldots,m$.
Compared to classical ADMM and QIPMs, our QADMM has a more favorable dimension dependence under the stated quantum data-access assumptions. A comparison of total runtime complexity of existing algorithms is summarized in Table~\ref{tab:comparison-of-total-runtime-complexity}.

\vspace{-0.1cm}
\begin{table}[h!]
    \centering
    \begin{tabular}{ll}
      \toprule
      \textbf{Algorithm} & \textbf{Complexity} \\
      \midrule
      QADMM (Algorithm~\ref{alg:QADMM-for-SDP})
        &    $
    \widetilde{\mathcal{O}}\left(n^2 ({\kappa_A^2}( 1 + R_y)^2 + R_X )\frac{(R_X+R_S)^3}{\epsilon_{\mathrm{abs}}^3}\right)
        $\\ 
      classical ADMM~\cite{wen2010alternating} 
         & \(\mathcal{O}\left(n^6 + n^4 \frac{R_X^2 + R_S^2}{\epsilon_{\mathrm{abs}}}\right)\)\\ 
      QMWU~\cite{van2018improvements, van2017quantum} 
        & \(\widetilde{\mathcal{O}}\left((n^2+n^{1.5}\,\tfrac{R_{\mathrm{Tr}X}R_y}{\epsilon_{\mathrm{abs}}})( \tfrac{R_{\mathrm{Tr}X }R_y}{\epsilon_{\mathrm{abs}}})^4\right)\) \\ 
      QIPM~\cite{mohammadisiahroudi2025quantum} 
          &\(\widetilde{\mathcal{O}}\left(n^{3.5}\kappa_{0}^2\right)\) \\
      \bottomrule
    \end{tabular}
    \vspace{0.3cm}
    \caption{Comparison of total runtime complexity (quantum gate complexity for quantum algorithms and classical arithmetic operations for classical ADMM) of Algorithms for SDP given that $m = \mathcal{O}(n^2)$ and the data matrices $A^{(i)}$ are fully dense.
    The notation $\epsilon_{\mathrm{abs}}$ denotes absolute accuracy in the objective value and the relevant feasibility residual, $R_{\mathrm{Tr}X}, R_X, R_y, R_S$ are the upper bounds of $\mathrm{Tr}(X^*)$, $\|X^*\|_{\mathrm{F}}$, $\|y^*\|_1$, and $\|S^*\|_{\mathrm{F}}$, respectively, and $\kappa_{A}, \kappa_{0}$ are the condition number of the data matrices and the initial IPM Newton system, respectively. QADMM enjoys a favorable scaling with respect to the dimension $n$ as compared to classical ADMM and QIPM. Compared to QMWU, QADMM 
    has the same order in dimension but achieves a lower-order dependence on the accuracy parameter. The comparison with QMWU depends on the size parameters of the SDP solution. QADMM is particularly favorable for instances where $\|X^\star\|_{\mathrm{F}}$ is much smaller than $\mathrm{Tr}(X^\star)$, for example diagonal SDP embeddings of linear programs or diagonally dominant perturbations of such instances. In these cases, $\mathrm{Tr}(X^\star)$ can scale linearly with $n$ while $\|X^\star\|_{\mathrm{F}}$ scales like $\sqrt{n}$, so the Frobenius-radius dependence in QADMM can be advantageous.
    }
    \label{tab:comparison-of-total-runtime-complexity}
  \end{table}

\subsection{Organization}
The remainder of this paper is organized as follows.
Section~\ref{sec:preliminaries} reviews the optimality conditions for SDPs and the quantum primitives used in our algorithm, including block-encodings, QRAM, tomography, QSVT-based eigenvalue transformation, and quantum linear solvers.
Section~\ref{sec:qadmm-for-sdp} presents the inexact ADMM framework for SDP, the QADMM algorithm, and the polynomial proximal operator used to approximate the semidefinite projection.
Section~\ref{sec:complexity-analysis} analyzes the iteration complexity and total implementation cost of QADMM.
Section~\ref{sec:numerical-simulation} reports numerical simulations for the operator-level QSVT implementation of the QADMM update.
Section~\ref{sec: conclusion} concludes the paper.
Technical details are presented in the Appendix.

\subsection{Notations}
Let $\mathbb{S}^n$ be the set of $n \times n$ symmetric matrices, $\mathbb{S}_+^n$ as the set of $n \times n$ positive semi-definite (PSD) matrices,
$- \mathbb{S}_+^n$ as the set of $n \times n$ negative semi-definite matrices.
The notations
$\mathrm{Proj}_{\mathbb{S}_+^n}(X)$ and $\mathrm{Proj}_{-\mathbb{S}_+^n}(X)$ denote the projection of $X$ onto $\mathbb{S}_+^n$ and $- \mathbb{S}_+^n$, respectively.
 Assume $X$ has the eigenvalue decomposition $X=U \Lambda U^{\top}$, where $U \in \mathbb{R}^{n \times n}$ is an orthogonal matrix and $\Lambda$ is a diagonal matrix with the eigenvalues of $X$ on the diagonal. Then $\mathrm{Proj}_{\mathbb{S}_+^n}(X)=U \Lambda^{+} U^{\top}$, $\mathrm{Proj}_{-\mathbb{S}_+^n}(X)=U \Lambda^{-} U^{\top}$,
 where $\Lambda^{+}$, $\Lambda^{-}$
are the diagonal matrices with the positive part and the negative part of the eigenvalues of $\Lambda$ on the diagonal, respectively.
The notation $\| \cdot \|$ (when subscripts are omitted) denotes the Frobenius norm of a matrix or 2-norm of a vector and $\langle \cdot, \cdot \rangle$ denotes the inner product of two matrices or vectors. 
In a Hilbert space $\mathcal{X}$, 
 the normal cone of a set $C \subset \mathcal{X}$ at a point $x$ is defined as $\mathcal{N}_C(x) = \{ v \in \mathcal{X} \mid \langle v, y - x \rangle \leq 0, \forall y \in C \}$. The proximal mapping of a function $h: \mathcal{X} \to \mathbb{R}$ is defined by
$
\mathrm{prox}_{\gamma h}^{\mathcal{X}}(V) = \arg \min_{S \in \mathcal{X}} \left\{ {h}(S) + \frac{1}{2 \gamma} \| S - V \|^2 \right\}
$. The superscript $\mathcal{X}$ is omitted when $\mathcal{X}$ is the full Euclidean space.
We denote the identity operator by $\mathrm{Id}$. For complex variables, we denote the real part and imaginary part of $x$ by $\mathrm{Re}(x)$ and $\mathrm{Im}(x)$, respectively.

\section{Preliminaries}
\label{sec:preliminaries}

\subsection{Optimality conditions}

For the primal-dual pair of SDPs, the optimality conditions
are given by the following theorem.
\begin{proposition}
    (\textbf{Optimality conditions})
    \label{prop: optimality conditions}
    $(X^*, y^*, S^*)$ is a primal-dual optimal solution of the primal-dual pair of SDPs~\eqref{eq: primal SDP} and \eqref{eq: dual SDP} if and only if the following conditions hold:
\begin{align}
  \mathcal{A}(X^*) &= b,\quad X^* \succeq 0,      \tag{Primal feasibility}\\
  \mathcal{A}^*(y^*) + S^* &= C,\quad S^* \succeq 0, \tag{Dual feasibility}\\
  X^* S^* &= 0.                                  \tag{Complementary slackness}
\end{align}
\end{proposition}
\subsection{Block-encodings and quantum random access memory}
\label{sec: qram and block encodings}
Block-encoding is a method in quantum computing that allows one to embed an arbitrary matrix into a unitary operator. In this framework, a matrix $A$ is said to be block-encoded in a unitary $U$ if there exists a normalization factor $\alpha$ (with $\alpha \ge \|A\|$) such that
$
U = \begin{pmatrix} \frac{A}{\alpha} & \cdot \\ \cdot & \cdot \end{pmatrix} .
$
A strict definition of block-encoding is as follows. 
\begin{definition}
    (\textbf{Block-encoding})
    \label{def: block-encoding}
Let $A \in \mathbb{C}^{d\times d}$ be an arbitrary matrix, and $U$ be a unitary operator acting on an extended Hilbert space $\mathbb{C}^{2^a}\otimes \mathbb{C}^{d}$. We say that $U$ is an $(\alpha, a, \epsilon)$-block-encoding of $A$ if
$$
\left\| A - \alpha\, \left( \langle 0^{\otimes a}| \otimes I_d \right) U \left( |0^{\otimes a}\rangle \otimes I_d \right) \right\| \le \epsilon,
$$
where $\alpha \ge \|A\|$ is a normalization factor, $I_d$ is the identity operator on the $d$-dimensional space, $a$ denotes the number of ancilla qubits, $\epsilon$ is the allowable error (with $\epsilon = 0$ corresponding to an exact block-encoding, and also termed as an $(\alpha, a)$-block-encoding).
\end{definition}
Quantum random access memory (QRAM) enables simultaneous querying of multiple addresses in superposition.  Concretely, suppose a QRAM stores classical vectors $v_j\in\mathbb{R}^{d}$ and an input register is prepared in the state $\sum_{j=0}^{2^w-1}\beta_j\lvert j\rangle$.  Under the QRAM assumption, one can implement the map
$$
\sum_{j=0}^{2^w-1}\beta_j\lvert j\rangle\otimes\lvert 0\rangle
\;\longrightarrow\;
\sum_{j=0}^{2^w-1}\beta_j\lvert j\rangle\otimes\lvert v_j\rangle,
$$
in time $\widetilde{\mathcal{O}}(1)$. Throughout the paper, QRAM is used as a data-access model for loading the input matrices, vectors generated by the classical part of the algorithm, and classical descriptions of iterates reconstructed by tomography. The storage required for these data structures is polynomial in the number of stored entries; it is not the exponential-size auxiliary memory sometimes appearing in older formulations of tomography. Under the QRAM assumption, we can 
implement the block-encodings of the data matrices 
$C, A^{(i)}, i=1, \ldots, m$ in~\eqref{eq: primal SDP} using the following proposition.
\begin{proposition}
    (\textbf{Block-encoding of matrices stored in QRAM})
    \label{prop: block-encoding of matrices stored in QRAM}
    (Lemma 50.2 in~\cite{gilyen2019quantum}). Let $A \in \mathbb{C}^{d \times d}$ and $\xi>0$. 
    If $A$ is stored in a QRAM data structure, then there exist unitaries $U_R$ and $U_L$ that can be implemented in time $\mathcal{O}\left(\operatorname{poly}\left(w \log \frac{1}{\xi}\right)\right)$ such that $U_R^{\dagger} U_L$ is an $(\alpha, w+2, \xi)$-block-encoding of $A$, where $\alpha = \|A\|_{\mathrm{F}}$ and $w = \lceil \log_2 d\rceil$.
\end{proposition}

\subsection{Tomography}
Quantum tomography is a technique used to reconstruct the quantum state of a system based on measurements. It is particularly useful for estimating the parameters of a quantum state when only partial information is available. The following proposition provides a quantum algorithm for estimating the parameters of a quantum state with high probability and low error.

\begin{proposition}(Theorem 2 in\cite{augustino2023quantum})
    \label{prop: tomography}
Let $|\psi\rangle=\sum_{j=0}^{d-1} y_j|j\rangle$ be a quantum state, $y \in \mathbb{C}^d$ the vector with elements $y_j$, and $U|0\rangle=|\psi\rangle$. There is a quantum algorithm that, with probability at least $1-\delta$, outputs $\tilde{y} \in \mathbb{R}^d$ such that $\|\mathrm{Re}(y)-\tilde{y}\|_2 \leq \varepsilon$ using $\widetilde{\mathcal{O}}\left(\frac{d}{\varepsilon} \right)$ applications of $U$ and $\widetilde{\mathcal{O}}\left(\frac{d}{\varepsilon}\right)$ indexed-SWAP or classical-write, quantum-read memory operations, together with $\widetilde{\mathcal{O}}\left(\frac{d}{\varepsilon}\right)$ additional gates. We use this linear-in-output-size form of tomography in the complexity analysis below.
\end{proposition}

\subsection{Eigenvalue transformation} \label{sec: eigenvalue transformation}

Given a real symmetric matrix $X \in \mathbb{S}^n$ with eigenvalue decomposition $X = U \Lambda U^{\top}$, where $U \in \mathbb{R}^{n \times n}$ is an orthogonal matrix and $\Lambda$ is a real diagonal matrix. 
Further given some function $g: \mathbb{R} \rightarrow \mathbb{R}$, we define the eigenvalue transformation operator 
\footnote{
    The operator    
$\boldsymbol{g}$ is also called a spectral operator in the literature of optimization~\cite{ding2018spectral}.
} with respect to $g$ using boldsymbol $\boldsymbol{g}$ as
$$
\boldsymbol{g}(X) = U \mathrm{diag}(g(\lambda_1), \ldots, g(\lambda_n)) U^{\top}.
$$
As a special case, $g(x) = \max\{0, x\}$ indicates $\boldsymbol{g}(X) = 
\mathrm{Proj}_{\mathbb{S}_+^n}(X)$, which can be computed as $U \Lambda^{+} U^{\top}$ with $\Lambda^{+} = \mathrm{diag}\left(\max\{0, \lambda_1\}, \ldots, \max\{0, \lambda_n\}\right)$.
For a general function $g$, 
the classical way to compute $g(X)$ requires computing the eigenvalue decomposition of $X$ and then applying the function $g$ to the eigenvalues. On a quantum computer, given the appropriate block-encoding access model, QSVT implements polynomial eigenvalue transformations with query complexity proportional to the polynomial degree, as shown in the following proposition.

\begin{proposition}
    \label{prop: quantum speedup for spectral operator}
    (Theorem 56 in~\cite{gilyen2019quantum})
    Suppose that $U$ is an $(\alpha, a, \varepsilon)$-encoding of a Hermitian matrix $A$. If $\delta \geq 0$ and $P_{\Re} \in \mathbb{R}[x]$ is a polynomial satisfying that
$\left|P_{\Re}(x)\right| \leq \frac{1}{2}, \forall x \in [-1, 1]$
Then there is a quantum circuit $\tilde{U}$, which is an $(1, a+2,4 \deg(P_{\Re}) \sqrt{\varepsilon / \alpha}+\delta)$-encoding of $P_{\Re}(A / \alpha)$, and consists of $\deg(P_{\Re})$ applications of $U$ and $U^{\dagger}$ gates, a single application of controlled-$U$ and $\mathcal{O}((a+1) \deg(P_{\Re}))$ other one- and two-qubit gates. Moreover a description of such a circuit with a classical computer can be computed in time $\mathcal{O}(\operatorname{poly}(\deg(P_{\Re}), \log (1 / \delta)))$.
\end{proposition}

\subsection{Quantum linear solver}
\label{sec: quantum linear solver}

If block-encoding of a Hermitian matrix $H$ is provided, we can use quantum linear solver (QLS) to solve $H {x} = {b}$ as provided in the following proposition.

\begin{proposition}
    \label{prop: quantum linear solver}
    (\textbf{Quantum linear solver}, Theorem 30, Corollary 32 in~\cite{chakraborty2018power})  Let $r \in (0, \infty), \kappa \geq 2$ and $H$ be a Hermitian matrix such that its non-zero eigenvalues lie in $[-1,-1 / \kappa] \cup [1 / \kappa, 1]$. Suppose that
$    
\xi=o\left(\frac{\delta}{\kappa^2 \log^3 \frac{\kappa^2}{\delta}}\right)
$
and $U$ is an $(\alpha, a, \xi)$-block-encoding of $H$ that can be implemented in time $T_U$. Suppose further that we can prepare a state $|v\rangle$ that is in the image of $H$ in time $T_v$. 

\noindent
1. For any $\delta$, we can output a state that is $\delta$-close to $H^{-1}|b\rangle /\left\|H^{-1} b\right\|$ in time
$$
\widetilde{\mathcal{O}}\left(\kappa\left(\alpha\left(a+T_U\right) +T_v\right)\right) .
$$

\noindent
2.For any $\delta$, we can output $\tilde{e}$ such that
$
(1-\delta) \| H^{-1}|b\rangle\|\leq \tilde{e} \leq(1+\delta)\| H^{-1}|b\rangle \|
$
in time
$$
\widetilde{\mathcal{O}}\left(\frac{\kappa}{\delta}\left(\alpha\left(a+T_U\right) +T_v\right) \right) .
$$
\end{proposition}

\section{Quantum alternating direction method of multipliers for SDP}
\label{sec:qadmm-for-sdp}
\subsection{Inexact ADMM framework for SDP}
ADMM is a method designed to solve constrained problems by breaking them into simpler subproblems. A comprehensive review of ADMM for general composite problem is provided in Appendix~\ref{app:sec: review of ADMM}.
For dual SDP~\eqref{eq: dual SDP}, we consider encoding the constraint $S \succeq 0$ via a penalty function
$
\boldsymbol{h}: \mathbb{S}^n \to \mathbb{R}
$.
When $\boldsymbol{h}(S) = \boldsymbol{h_0}(S) := \delta_{\mathbb{S}_+^n}(S)$,~\eqref{eq: dual SDP 2} is equivalent to~\eqref{eq: dual SDP}. Smooth functions are also considered such as $\boldsymbol{h}(S) = \mu \log\det(S)$ in interior point methods, where $\mu > 0$ is a parameter that controls the duality gap. 
Denote 
$\mathcal{X} := \{X \in \mathbb{S}^n| X \succeq 0, \|X\|_{\mathrm{F}} \leq R_X\}$, $\mathcal{Y} := \{y \in \mathbb{R}^{m}| \|y\|_1 \leq R_y\}$, $\mathcal{S} := \{S \in \mathbb{S}^n| S \succeq 0, \| S \|_{\mathrm{F}} \leq R_S\}$ as the feasible sets of $X$, $y$ and $S$, respectively.
Define $f(y) := -b^T y$. 
Then~\eqref{eq: dual SDP} is rewritten as 
\begin{equation}
\label{eq: dual SDP 2}
\min_{y\in\mathcal{Y},S \in \mathcal{S}}  f(y) + \boldsymbol{h}(S) \qquad \text{s.t.}\quad A^*(y) + S = C~.
\end{equation}
This formulation is amenable to ADMM splitting by associating $f(y)$ with the $y$-update and $\boldsymbol{h}(S)$ with the $S$-update and treating the equation $A^*(y)+S=C$ as the coupling constraint.
The augmented Lagrangian function for the dual SDP is
\begin{equation}
    \label{eq: augmented Lagrangian for SDP}
\mathcal{L}_\gamma(X, y, S):=-b^{\top} y+ \boldsymbol{h}(S) + \langle X, \mathcal{A}^*(y)+S-C\rangle+\frac{1}{2 \gamma}\left\|\mathcal{A}^*(y)+S-C\right\|_{\mathrm{F}}^2,
\end{equation}
with some fixed $\gamma > 0$.
With this augmented Lagrangian function, ADMM iterates as follows~\cite{wen2010alternating}:
\begin{subequations}
    \begin{align}
         y^{k+1} & := \arg \min _{y \in \mathcal{Y}} \mathcal{L}_\gamma\left(X^k, y, S^k\right) =-\left(\mathcal{A A}^*\right)^{-1}(\gamma(\mathcal{A}(X^k)-b)+\mathcal{A}(S^k-C)), \label{eq: y in classical ADMM for SDP} \\
         S^{k+1} & := \arg \min _{S \in \mathcal{S}} \mathcal{L}_\gamma\left(X^k, y^{k+1}, S\right) =\mathrm{prox}_{\gamma \boldsymbol{h}}^{\mathcal{S}}\left(C-\mathcal{A}^*(y^{k+1})-\gamma X^k \right), \label{eq: S in classical ADMM for SDP}\\
         X^{k+1} & := X^k+\frac{1}{\gamma}(\mathcal{A}^*\left(y^{k+1}\right)+S^{k+1}-C) 
          = - \frac{1}{\gamma}(\mathrm{Id} - \mathrm{prox}_{\gamma \boldsymbol{h}}^{\mathcal{S}})\left(C-\mathcal{A}^*(y^{k+1})-\gamma X^k\right) 
         .
         \label{eq: X in classical ADMM for SDP} 
    \end{align}
\end{subequations}
The matrix $\mathcal{A A}^* \in \mathbb{R}^{m \times m}$ in~\eqref{eq: y in classical ADMM for SDP} is a symmetric positive definite matrix with its $(i, j)$-th element being $\langle A^{(i)}, A^{(j)}\rangle$. 
When $\boldsymbol{h}(S) = \boldsymbol{h_0}(S)$, the proximal operator $\mathrm{prox}_{\gamma \boldsymbol{h_0}}^{\mathcal{S}} = \mathrm{Proj}_{\mathcal{S}}$.
We now introduce an inexact ADMM framework for SDP, which allows errors in $y$-update and $S$-update.
By introducing an intermediate variable
$
    V^{k+1} = C-\mathcal{A}^*(y^{k+1})-\gamma X^k
$ and error terms, 
~\eqref{eq: y in classical ADMM for SDP}-\eqref{eq: X in classical ADMM for SDP} are splitted into several steps:
\begin{subequations}
    \begin{align}
        \hat y^{k+1} &=-\left(\mathcal{A A}^*\right)^{-1}(\gamma(\mathcal{A}(X^k)-b)+\mathcal{A}(S^k-C)), \label{eq:qadmm-hat-y} \\
         \tilde y^{k+1} & := \hat y^{k+1} + \Delta \tilde y^{k+1} \label{eq:qadmm-tilde-y}, \quad  \|\Delta \tilde y^{k+1}\|_1 \leq \delta_{\tilde y}, \\
         y^{k+1} & := \mathrm{Proj}_{\mathcal{Y}}(\tilde y^{k+1} + \Delta y^{k+1}) \label{eq:qadmm-y}, \quad  \|\Delta y^{k+1}\|_1 \leq \delta_y, \\
        \hat V^{k+1} &:= C-\mathcal{A}^*(\tilde y^{k+1})-\gamma X^k, \label{eq:qadmm-hat-V} \\
        \tilde V^{k+1} & := \hat V^{k+1} + \Delta V^{k+1} \label{eq:qadmm-tilde-V} , \quad  \|\Delta V^{k+1}\|_{\mathrm{F}} \leq \delta_V, \\
        \tilde S^{k+1} & := \mathrm{prox}_{\gamma \boldsymbol{h}}(\tilde V^{k+1}) ,\label{eq:qadmm-tilde-S} \\
       S^{k+1} & := \mathrm{Proj}_{\mathcal{S}} (\tilde S^{k+1} + \Delta S^{k+1}) \label{eq:qadmm-S}, \quad  \|\Delta S^{k+1}\|_{\mathrm{F}} \leq \delta_S, \\
       \tilde X^{k+1} & := -\frac{1}{\gamma}(\mathrm{Id} - \mathrm{prox}_{\gamma \boldsymbol{h}})(\tilde V^{k+1}) ,\label{eq:qadmm-tilde-X} \\
       X^{k+1} & := \mathrm{Proj}_{\mathcal{X}}(\tilde X^{k+1} + \Delta X^{k+1}) \label{eq:qadmm-X}, \quad  \|\Delta X^{k+1}\|_{\mathrm{F}} \leq \delta_X.
    \end{align}
\end{subequations}
The variables with tilde notations $\tilde y^{k+1}, \tilde{S}^{k+1}, \tilde{X}^{k+1}$ stand for quantum representations and variables without any notation $y^{k+1}, S^{k+1}, X^{k+1}$ stand for classical representations. Additionally, variables with hat notations $\hat y^{k+1}, \hat V^{k+1}$ are intermediate variables.
The error terms $\Delta \tilde y^{k+1}, \Delta y^{k+1}, \Delta V^{k+1}, \Delta S^{k+1}, \Delta X^{k+1}$ in~\eqref{eq:qadmm-tilde-y}-\eqref{eq:qadmm-X}
are used to account for the inexactness, which are bounded by $\delta_{\tilde y}, \delta_y, \delta_V, \delta_S, \delta_X$, respectively.

\subsection{QADMM for SDP}
The core idea of QADMM is to alleviate the computational bottlenecks of classical solvers by substituting numerically intensive operations with quantum subroutines.
The resulting algorithm is a quantum-classical hybrid algorithm by design. The classical processor evaluates inexpensive linear maps, simple bounded-set projections, normalization factors, and the classical descriptions needed for output and for the next iteration. The quantum processor is used for the two bottleneck operations: solving the linear system in the $y$-update and applying polynomial spectral transformations in the $S$- and $X$-updates. In particular, the PSD spectral component in the $S$- and $X$-updates is supplied by QSVT, while the classical post-processing only enforces simple radius bounds and stores the reconstructed iterates.
Intuitively, the penalty function $\boldsymbol{h}(S)$ in~\eqref{eq: dual SDP 2} is chosen
so that the form of its proximal operator $\mathrm{prox}_{\gamma h}(V)$ admits an efficient quantum implementation. 
The main procedure of the QADMM algorithm for SDP is described as follows. 

\noindent
1.
$y$-update~\eqref{eq: y in classical ADMM for SDP}.
First we compute 
\begin{equation}
\label{eq:qadmm-u}
u^{k+1}= \gamma(\mathcal{A}(X^k)-b)+\mathcal{A}(S^k-C),
\end{equation}
classically and produce a corresponding quantum state $ | \hat{u}^{k+1} \rangle  = \frac{u^{k+1}}{\|u^{k+1}\|_2 }$. 
Since $u^{k+1}\in\mathbb{R}^m$ has already been computed classically, preparing $|\hat u^{k+1}\rangle$ is done by loading this vector into the QRAM/state-preparation data structure and normalizing it. This contributes $\mathcal{O}(m)$ classical writes and standard quantum-read access per iteration, which is included in the classical data-loading cost and does not require an additional SDP-scale eigendecomposition.
Next we invoke QLS reviewed in Section~\ref{sec: quantum linear solver} to solve $ \mathcal{A}\mathcal{A}^* \frac{\hat y^{k+1}}{\| u^{k+1}\|} = - | \hat u^{k+1} \rangle$
thereby obtaining $ | v_y^{k+1} \rangle, e_y^{k+1}$ such that
\begin{equation}
    \label{eq: error bound for y}
\left\| | v_y^{k+1} \rangle - \frac{\hat y^{k+1}}{\|\hat y^{k+1}\|_2} \right\|_2 \leq \delta_{v_y}, \quad
( 1 - \delta_{e_y}) \left\| \hat{y}^{k+1} \right\|_2 \leq  e_y^{k+1} \leq ( 1 + \delta_{e_y}) \left\| \hat{y}^{k+1} \right\|_2,
\end{equation}
where $\delta_{v_y}, \delta_{e_y} > 0$.
This yields the quantum representation  $\tilde{y}^{k+1} = e_y^{k+1} | v_y^{k+1} \rangle \approx \hat y^{k+1}$ and the error is bounded by $\delta_{\tilde y} = \delta_{e_y} + \delta_{v_y} \|\hat y^{k+1}\|_2$, i.e.,
\begin{equation}
    \label{eq: error bound Delta y}
\Delta y^{k+1} 
\leq \left\| e_y^{k+1} | v_y^{k+1} \rangle - \|\hat y^{k+1}\|_2 | v_y^{k+1} \rangle \right\|_2 + \left\| \| \hat y^{k+1}\|_2 | v_y^{k+1} \rangle  - \hat y^{k+1} \right\| \leq \delta_{e_y} + \delta_{v_y}\| \hat y^{k+1}\|_2.
\end{equation}
Finally, we perform tomography on $\tilde y^{k+1}$ to reconstruct a classical vector and then project it onto the feasible set $\mathcal{Y}$.

\noindent
2. 
$V$-update~\eqref{eq:qadmm-tilde-y}-\eqref{eq:qadmm-hat-V}.
Equation~\eqref{eq:qadmm-hat-V} is evaluated on the quantum computer via the linear-combination-of-unitaries (LCU) technique.  At the start of iteration $k$, the matrix $X^{k}$ already resides in quantum memory. Let the 
$(\alpha_i, \lceil \log(n) \rceil + 2, \xi)$-block-encodings of $A^{(1)}, \ldots, A^{(m)}, C, X^{k}$ be denoted by $U_{A^{(i)}}, U_C, U_{X^{k}}$, where $\alpha_i = \|A^{(i)}\|_{\mathrm{F}}$ and $\alpha_{m+1} = \|C\|_{\mathrm{F}}$, $\alpha_{m+2} = \|X^k\|_{\mathrm{F}}$.
Then 
by the LCU lemma (Lemma 7.9 in~\cite{lin2022lecture}), we compute a block-encoding of the matrix
\begin{equation}
\label{eq: lcu}
\frac{1}{\sum_{i=1}^{m} \alpha_i | y^{k+1}_i| + \alpha_{m+1} +\gamma \alpha_{m+2}} \left(\sum_{i=1}^m \alpha_i y^{k+1}_i U_{A^{(i)}}  + 
\alpha_{m+1} U_{C} - \gamma \alpha_{m+2}  U_{X^k}\right)
\end{equation}
This produces the quantum representation $\tilde{V}^{k+1} = e_V^{k+1} v_V^{k+1}$.

\noindent Here $e_V^{k+1} = \sum_{i=1}^{m} \alpha_i |y^{k+1}_i| \allowbreak + \alpha_{m+1} \allowbreak + \gamma \alpha_{m+2}$ and
\[
v_V^{k+1}
= \frac{1}{e_V^{k+1}}\left( \sum_{i=1}^m y^{k+1}_i A^{(i)} + C - \gamma X^k \right).
\]
The approximation error incurred by the block-encoding and LCU is denoted by $\Delta V^{k+1}$ in~\eqref{eq:qadmm-tilde-V}.

\noindent
3. 
$S$-update~\eqref{eq:qadmm-tilde-V}-\eqref{eq:qadmm-S}.
We replace $\mathrm{prox}_{\gamma \boldsymbol{h}}$ in~\eqref{eq:qadmm-tilde-S} with a polynomial matrix function.
Let $g_1: [-1, 1] \to\mathbb{R}$ be a polynomial of degree $d$ and $\boldsymbol{g}_1$ be the corresponding spectral operator. Choose a constant $c_1\geq \|g_1\|_{\infty,[-1,1]}$ with $c_1=\Theta(1)$ and set $p_1=g_1/(2c_1)$. 
The step~\eqref{eq:qadmm-tilde-S} is computed by a spectral transformation of the form 
\begin{equation}
    \label{eq: proximal operator}
\tilde S^{k+1} =  e_V^{k+1} \boldsymbol{g}_1( v_V^{k+1}) = 2c_1 e_V^{k+1} \boldsymbol{p}_1(v_V^{k+1}) .
\end{equation}
The block-encoding of $v_V^{k+1}$ is computed in the previous step.
Crucially, we will compute the scaled polynomial $\boldsymbol{p}_1$ using QSVT introduced in Section~\ref{sec: eigenvalue transformation}; the final factor $2c_1e_V^{k+1}$ is restored in the classical reconstruction. The complete procedure of $S$-update is as follows:
Using QRAM and QSVT, we can design a quantum circuit to compute the block-encoding of $\boldsymbol{p}_1(v_V^{k+1})$.
Tomography then reconstructs a classical approximation of $2c_1e_V^{k+1}\boldsymbol{p}_1(v_V^{k+1})=e_V^{k+1}\boldsymbol{g}_1(v_V^{k+1})$, which is the intermediate matrix $\tilde S^{k+1}$ in~\eqref{eq:qadmm-tilde-S}. The notation $S^{k+1}=\mathrm{Proj}_{\mathcal{S}}(\tilde S^{k+1}+\Delta S^{k+1})$ in~\eqref{eq:qadmm-S} is used in the convergence analysis to encode boundedness and feasibility of the ideal iterate. In the quantum implementation, the PSD component is approximated by the QSVT polynomial spectral map; after tomography one stores the reconstructed matrix and, if needed, applies only inexpensive Frobenius-radius clipping. The deviation from the exact PSD projection is counted in $\Delta \tilde S^{k+1}$ and $\Delta S^{k+1}$, so no additional classical eigendecomposition is hidden in this step.

4. $X$-update~\eqref{eq:qadmm-tilde-X}-\eqref{eq:qadmm-X}. This step proceeds analogously to the S-update.
Define $\boldsymbol{g}_2$ as a polynomial replacing $\mathrm{Id} - \mathrm{prox}_{\gamma \boldsymbol{h}}$ in~\eqref{eq:qadmm-tilde-X}; for QSVT implementation we use the same constant rescaling $p_2=g_2/(2c_2)$ with $c_2=\Theta(1)$.
We compute $V^{k+1}$ via step 2, then employ QSVT and tomography to reconstruct $X^{k+1}$. 
Finally, we use QRAM to store $X^{k+1}$ in quantum memory to serve as input for $V$-update in the subsequent iteration.
The complete procedure is summarized in Algorithm~\ref{alg:QADMM-for-SDP}. Implementation details are provided in Appendix~\ref{app:sec: implementation details for QADMM}.

\begin{algorithm}[!htbp]
    \centering
    \caption{QADMM for SDP}
    \label{alg:QADMM-for-SDP}
    \begin{algorithmic}
        \STATE \textbf{Input:} Data $C, b, {A}^{(i)}, i=1,\ldots, m$.
        \STATE Initialize $(X^0, y^0, S^0)$.
        \FOR{$k=0,1,2,\ldots, K-1$}
        \STATE (1) (y-update) 
        Compute $u^{k+1}$ by~\eqref{eq:qadmm-u} classically and prepare a quantum representation of $u^{k+1}$.
        \STATE \quad \quad Solve the linear system $\tilde y^{k+1} = -\left(\mathcal{A A}^*\right)^{-1} u^{k+1}$ on the quantum computer.
         \STATE \quad \quad Construct classical $ y^{k+1}$ satisfying~\eqref{eq:qadmm-y} using vector state tomography.
        \STATE (2) (V-update)   
        \STATE Compute $V^{k+1}$ by~\eqref{eq:qadmm-tilde-V} via LCU on the quantum computer.

        \STATE (3) (S-update)
        \STATE Compute $\tilde S^{k+1}$ by~\eqref{eq:qadmm-tilde-V}-~\eqref{eq:qadmm-tilde-S} via QSVT on the quantum computer.
        \STATE \quad \quad Construct classical $ S^{k+1}$ satisfying~\eqref{eq:qadmm-S} using tomography.
        \STATE (4) (X-update)  Update $\tilde X^{k+1}$ by~\eqref{eq:qadmm-tilde-X} via QSVT on the quantum computer.
        \STATE \quad \quad Construct classical $ X^{k+1}$ satisfying~\eqref{eq:qadmm-X} using tomography.
        \STATE \quad \quad Recompute and store $\tilde X^{k+1}$ on the quantum computer.
        \ENDFOR 
        \STATE \textbf{Output:} $(X^K, y^K, S^K)$ or  $(X_{out}, y_{out}, S_{out}) := \left(\frac{1}{K} \sum_{k=1}^K X^{k}, \frac{1}{K} \sum_{k=1}^K y^{k}, \frac{1}{K} \sum_{k=1}^K S^{k}\right)$ .
    \end{algorithmic}
\end{algorithm}

\subsection{Polynomial proximal operator} \label{sec: polynomial barrier for SDP}

In this section, we detail the construction of the polynomial  $\boldsymbol{g}_1, \boldsymbol{g}_2$ introduced in the $S$-update and $X$-update. 
Without loss of generality, we focus on $\boldsymbol{g}_1$; the treatment  of $\boldsymbol{g}_2$ is entirely analogous.
The goal is to choose
$\boldsymbol{g}_1$ as an approximation of the projection operator $\mathrm{Proj}_{\mathbb{S}_+^n}(V)$, which preserve the convergence of the inexact ADMM framework. At the same time, $\boldsymbol{g}_1$ is a polynomial transformation, ensuring that it can be implemented efficiently on a quantum computer.
The following lemma shows the existence of such polynomials and their detailed proofs are given in Appendix~\ref{sec: polynomial barrier for SDP}.
\begin{lemma}
\label{lem: polynomial approximation}
1. For any $\epsilon > 0$, there exists a monotone increasing polynomial $g(x)$ on $[-1, 1]$ of degree $d = \mathcal{O}(\frac{1}{\epsilon})$ such that $g(-1)= 0$ and
\begin{equation}
    \label{eq: polynomial approximation}
     | \max(0, x) - g(x) | \leq \epsilon, \quad \forall x \in [-1, 1].
\end{equation}

2. Let $\boldsymbol{g}$ be the corresponding spectral operator of $g$.
Then $\boldsymbol{g}$ is a polynomial eigenvalue transformation on $\mathbb{S}^n$ and satisfies
\begin{equation}
    \label{eq: polynomial approximation for SDP}
     \| \mathrm{Proj}_{\mathbb{S}_+^n}(X) - \boldsymbol{g}(X) \|_2 \leq \epsilon, \quad \forall X \in 
     \{ X \in \mathbb{S}^n: \| X \|_2 \leq 1\}  .
\end{equation}
\end{lemma}
Equivalently, choose a constant $c_1\geq \|g_1\|_{\infty,[-1,1]}$ with $c_1=\Theta(1)$ and set $p_1=g_1/(2c_1)$. Since $|p_1(x)|\le 1/2$ on $[-1,1]$, the scaled spectral operator $\boldsymbol{p}_1$ can be implemented via QSVT in accordance with Proposition \ref{prop: quantum speedup for spectral operator}; multiplying the reconstructed output by $2c_1$ recovers $\boldsymbol{g}_1$.
Although a convex $\boldsymbol{h}$ satisfying $\mathrm{prox}_{\gamma \boldsymbol{h}}^{\mathcal{S}} = \boldsymbol{g}_1$ may not exist since $g$ is not nonexpansive, we can still choose $\boldsymbol{h}$ satisfying the variational equation $(\mathrm{Id} + \gamma \partial \boldsymbol{h})^{-1}(S) = \boldsymbol{g}_1(S)$.
This variational identity explains which first-order optimality equation the polynomial update approximately satisfies, while the actual convergence proof below compares the polynomial update directly with the PSD projection and treats the difference as an inexactness term.
Strict monotonicity is used only for this variational interpretation: if the polynomial approximation is merely monotone on $[-1,1]$, one may add an arbitrarily small strictly increasing perturbation and extend it smoothly to $\mathbb R$; the additional error is absorbed into $\delta_{\tilde S}$ and $\delta_{\tilde X}$.

\begin{lemma}
    \label{lem: proximal operator construction}
    Let $g: \mathbb{R} \to \mathbb{R}$ be a smooth function satisfying $g'(x)>0$ for all $x \in \mathbb{R}$. Denote $g^{-1}$ as the inverse function of $g$.
    Then 
    \begin{equation}
        \label{eq: proximal operator construction for scalar}
    h(x) = 
    \begin{cases}
        \frac{1}{\gamma} \int_{0}^{x} (g^{-1}(z) - z) dz , & \text{if } x \geq 0, \\
        \infty, & \text{if } x < 0,    
    \end{cases}
    \end{equation}
    satisfies $(\mathrm{Id}+\gamma\partial h)^{-1}(x)=g(x)$.
    
    \noindent
    2. 
    For $X \in \mathbb{S}^n$ with eigenvalue decomposition $X = U \mathrm{diag}(\lambda_1, \ldots, \lambda_n) U^\top$, where $U$ is a unitary matrix and $\lambda_i$ are the eigenvalues of $X$.
    Let $\boldsymbol{g}(X) = U \mathrm{diag}(g(\lambda_1), \ldots, g(\lambda_n)) U^\top$ be the eigenvalue transformation with respect to $g$.
    Define $\boldsymbol{h}: \mathbb{S}^n \to \mathbb{R}$ as
    $\boldsymbol{h}(X) = \sum_{i=1}^{n} h(\lambda_i)$, where $h$ is defined in~\eqref{eq: proximal operator construction for scalar}.
    Then $\boldsymbol{h}$ satisfies $(\mathrm{Id}+\gamma\partial \boldsymbol{h})^{-1}(S)=\boldsymbol{g}(S)$.
\end{lemma}

By Lemma~\ref{lem: polynomial approximation} and~\ref{lem: proximal operator construction}, a ``virtual'' barrier function $\boldsymbol{h}(S)$ is designed such that $\boldsymbol{g}_1 = (\mathrm{Id} + \gamma\partial \boldsymbol{h})^{-1}$ is a polynomial eigenvalue transformation. 
The possible nonconvexity of $\boldsymbol{h}$ therefore does not enter as a nonconvex ADMM objective. We do not rely on global minimization of a nonconvex proximal subproblem. Instead, the QSVT step implements a polynomial map $\boldsymbol{g}_1$ that is uniformly close to $\mathrm{Proj}_{\mathbb{S}_+^n}$ on the bounded spectral interval, and the convergence theorem uses the resulting errors $\Delta \tilde S^{k+1}$ and $\Delta \tilde X^{k+1}$ as controlled inexactness terms.
We denote their per-iteration bounds by $\delta_{\tilde S}$ and $\delta_{\tilde X}$, respectively.

\section{Complexity analysis}
\label{sec:complexity-analysis}

In this section, we analyze the complexity of QADMM for SDP. We first establish the convergence of QADMM for SDP and then derive the iteration complexity, and finally evaluate its overall computational cost. The detailed proofs of the results in this section are presented in Appendix~\ref{sec: complexity of QADMM for SDP}.

\subsection{Iteration number of QADMM for SDP}

\newcommand{\nonergodicconvergence}{
    (non-ergodic convergence)
The sequence $\{X^k, y^k, \tilde S^k, S^k\}$ generated by the scheme~\eqref{eq: y in ADMM for SDP}--\eqref{eq: X in ADMM for SDP} satisfies
\begin{subequations}
    \begin{align}
        - b^\top y^K  + b^\top y^* + \boldsymbol{h}(S^K) - \boldsymbol{h}(S^*)
        &\leq O\left(\sqrt{\frac{\gamma D}{K} + (K-1)\gamma e^{(1)} +2\gamma e^{(2)}}\right), \label{eq: non-ergodic convergence 1} \\
        \left\|\mathcal{A}^*(y^K) + S^K - C\right\|  &\leq \sqrt{\frac{\gamma D}{K} + (K-1)\gamma e^{(1)} +2\gamma e^{(2)}}. 
        \label{eq: non-ergodic convergence 2}
        \end{align}
\end{subequations}
}

In this subsection, we analyze the ergodic convergence of the inexact QADMM scheme. 
Without loss of generality, assume $\gamma = 1$.
We show that the averaged iterate
converges to an $\epsilon$-optimal solution of~\eqref{eq: dual SDP} in the sense of the objective value and the dual feasibility.

\begin{theorem}
    \textbf{(ergodic convergence)}
    \label{thm: ergodic convergence of QADMM for SDP}
    Assume $(X^*, y^*, S^*)$ satisfies the optimality conditions.
    The average iterate $(\bar{X}^K, \bar{y}^K, \bar{S}^K) := \frac{1}{K} \sum_{k=0}^{K-1} (\tilde X^k, \hat y^k, \tilde S^k)$  
     generated by the scheme~\eqref{eq:qadmm-hat-y}-\eqref{eq:qadmm-X} satisfies
    \begin{subequations}
    \begin{align}
    f\left(\bar{y}^K\right)-f\left(y^*\right) & \leq \mathcal{O}\left(\frac{\left\|X^0 - X^*\right\|_{\mathrm{F}}^2+\left\| S^0-S^*\right\|_{\mathrm{F}}^2}{K} +\delta \right)\label{eq: ergodic convergence of QADMM for SDP 1}, \\
    \left\|\mathcal{A}^*( \bar{y}^K ) + \bar{S}^K-C\right\|_{\mathrm{F}} & \leq \mathcal{O}\left(\frac{\left\|X^0-X^*\right\|_{\mathrm{F}}+\left\|S^0-S^*\right\|_{\mathrm{F}}}{K}+\delta\right) , \label{eq: ergodic convergence of QADMM for SDP 2} 
    \end{align} 
    \end{subequations}
     where $\delta = \Theta(\delta_{\tilde y} + \delta_y + \delta_V +\delta_X + \delta_S + \delta_{\tilde S} + \delta_{\tilde X})$ is the accumulated ADMM inexactness term. The QLS accuracy parameter $\delta_{\hat y}$ is chosen so that the induced iterate error satisfies $\delta_{\tilde y}=\Theta(\delta_{\hat y})$.
\end{theorem}
However, the quantum representations of variables are not accessible directly. For the classical representations reconstructed via tomography, we have the following results.
\begin{corollary}
\label{cor: ergodic convergence of QADMM for SDP 2}
Let the initial point be $(X^0, y^0, S^0) = (0, 0, 0)$. Then the output of Algorithm~\ref{alg:QADMM-for-SDP} satisfies
\begin{equation}
\begin{aligned}
f({y}_{out}) - f(y^*) &
\leq \mathcal{O}\left(\frac{R_X^2+R_S^2}{K} + \delta \right) ,
\\
 \left\|\mathcal{A}^*( y_{out} ) + S_{out}-C\right\|_{\mathrm{F}} & \leq
 \mathcal{O}\left(\frac{R_X+R_S}{K} + \delta  \right) .
\end{aligned} 
\end{equation}
\end{corollary}
Combining Theorem~\ref{thm: ergodic convergence of QADMM for SDP} and Corollary~\ref{cor: ergodic convergence of QADMM for SDP 2}
, we derive the iteration complexity as follows.
\begin{corollary}
    Assume $ \delta = \Theta(\epsilon_{\mathrm{abs}})$ and let the initial point be $(X^0, y^0, S^0) = (0, 0, 0)$.
    To achieve an $\epsilon_{\mathrm{abs}}$-accuracy solution $(y_{out}, S_{out})$ satisfying
    $
    - b^\top y_{out} + b^\top y^* \leq \epsilon_{\mathrm{abs}}
    $ and $\|A^*(y_{out}) + S_{out} - C\|_{\mathrm{F}} \leq \epsilon_{\mathrm{abs}}$,
    the algorithm~\ref{alg:QADMM-for-SDP} requires 
    $$
    K = \mathcal{O}\left( \frac{R_X^2 + R_S^2}{\epsilon_{\mathrm{abs}}} \right) 
    $$
    iterations, where $R_X$ and $R_S$ are the upper bounds of $\| X^*\|_{\mathrm{F}}$ and $\| S^* \|_{\mathrm{F}}$, respectively.
\end{corollary}

\subsection{Complexity of QADMM for SDP}

In this section, we formally analyze the worst case overall running times of QADMM for SDP. The per-iteration computational cost is written as
$$
T_{\mathrm{iter}} = T_{\mathrm{quant}} +T_{\mathrm{classic}},
$$
where $T_{\mathrm{quant}}$ is the quantum gate complexity and $T_{\mathrm{classic}}$ counts the classical arithmetic operations. 

On a classical computer, the linear operator $\mathcal{A}$ defined in~\eqref{eq: linear map A} can be represented as a matrix $\boldsymbol{A} \in \mathbb{R}^{m \times n(n+1)/2}$ and the matrix $X$ is compressed as a vector of length $n(n+1)/2$ for computing $\mathcal{A}(X)$.
Assume $\boldsymbol{A} \in \mathbb{R}^{m \times n(n+1)/2}$ is a sparse matrix with $s$ non-zero entries. 
The cost of classically computing $\mathcal{A}(X)$ is $T_{A} = \mathcal{O}(s_A + n^2)$ 
and computing $\mathcal{A}^*(y)$ is $T_{A^*} = \mathcal{O}(s_A + n^2)$. As a special case, if $\boldsymbol{A}$ is a dense matrix, then $T_{A} = T_{A^*} = \mathcal{O}(m n^2)$. 
The complexity of each step in Algorithm~\ref{alg:QADMM-for-SDP} is analyzed as follows:

\noindent
1.
		$y$-update. Assume the condition number of $\mathbf{A}\mathbf{A}^\top$ is $\kappa_A^2$.
	    By Proposition~\ref{prop: quantum linear solver} and~\eqref{eq: error bound Delta y}, QLS requires a runtime of $T_1 = \widetilde{\mathcal{O}}(\frac{\kappa_A^2(1+ \|\hat y\|_2)}{\delta_{\hat y}}) \leq \widetilde{\mathcal{O}}(\frac{\kappa_A^2(1+ R_y)}{\delta_{\hat y}})$ to achieve accuracy $\delta_{\hat y}$ and the tomography step requires $\mathcal{O}(m \frac{R_y}{\delta_y})$ queries to achieve accuracy $\delta_y$. 
	    Additionally, $\mathcal{O}(T_{A})$ classical operations are required for computing linear mappings. Preparing $|\hat u^{k+1}\rangle$ from the already computed vector $u^{k+1}$ requires $\mathcal{O}(m)$ classical writes into the state-preparation data structure; under the dense regime $m=\mathcal{O}(n^2)$ used in Table~\ref{tab:comparison-of-total-runtime-complexity}, this is dominated by the stated $\mathcal{O}(s_A+n^2)$ classical cost.

\noindent
2.	V-update. In this step,
LCU requires a runtime of $T_2 = \widetilde{\mathcal{O}}(e_V^{k+1}) \leq \widetilde{\mathcal{O}}(1 + R_X + R_y)$.

\noindent
3.	S-update.
By Proposition~\ref{prop: quantum speedup for spectral operator} and Lemma~\ref{lem: polynomial approximation}, applying QSVT to evaluate the polynomial transformation requires $\mathcal{O}(d)$ quantum gates, $\mathcal{O}(d)$ oracle queries, and $\mathcal{O}(1)$ ancillary qubits
to achieve that $\| \tilde{S}^{k+1} - \mathrm{Proj}_{\mathbb{S}_+^n}(\tilde V^{k+1})\|_2 \leq \epsilon$, where $d = \mathcal{O}(\frac{e_V^{k+1}}{\epsilon})
\leq \mathcal{O}(\frac{1 + R_X + R_y}{\epsilon})$ is the degree of the polynomial transformation.
We then construct a classical description of the iterate by performing state tomography on $\tilde S^{k+1}$.  Proposition~\ref{prop: tomography} implies that tomography requires $\mathcal{O}(n^2 \frac{R_S}{\delta_{S}})$ queries on the block-encoding of $\tilde S^{k+1}$ to achieve accuracy $\delta_S$. Hence the total runtime is $\widetilde{\mathcal{O}}( (T_1 + T_2 + d) \frac{n^2 R_S}{\delta_{S}})$.

\noindent
4.
$X$-update. Similarly to $S$-update, this step has a runtime of $O( (T_1 + T_2 +d) \frac{n^2 R_X}{\delta_{X}})$.

Summarizing the above analysis, one single QADMM iteration incurs
$
T_{\mathrm{classic}}= \mathcal{O}( s_A + n^2)
$
classical operations, and 
$$
T_{\mathrm{quant}} = \widetilde{\mathcal{O}} \left( \frac{\kappa_A^2(1+R_y) }{\delta_{\hat y}}m \frac{R_y}{\delta_y} + ( \frac{\kappa_A^2(1+R_y)}{\delta_{\hat y}}  + \frac{1 + R_X + R_y}{\epsilon}) (\frac{n^2 R_S}{\delta_{S}} + \frac{n^2 R_X}{\delta_{X}} ) \right)
$$
quantum gates and queries. 
Combining these per-iteration costs with the iteration bound established above yields the overall computational complexity of QADMM. This theorem no longer assumes a QRAM of size $\widetilde{\mathcal{O}}(2^{n^2/\epsilon_{\mathrm{abs}}})$; that term came from an overly literal restatement of a tomography primitive and is not needed for the algorithmic data access used here.
\begin{theorem}
    \label{thm: complexity of QADMM for SDP}
    With the accuracy choice ${\delta_{\hat y}, \delta_{\tilde y}, \delta_y, \delta_V, \delta_S, \delta_X, \delta_{\tilde S}, \delta_{\tilde X}} = \Theta(\epsilon_{\mathrm{abs}})$ and $d = \mathcal{O}(\frac{1+R_X+R_y}{\epsilon_{\mathrm{abs}}})$,
    a quantum implementation of Algorithm~\ref{alg:QADMM-for-SDP} under the polynomial-size QRAM data-access model described in Section~\ref{sec: qram and block encodings} and the block-encoding assumptions in Proposition~\ref{prop: block-encoding of matrices stored in QRAM} produces an $\epsilon_{\mathrm{abs}}$-accuracy solution using at most $\mathcal{O}\left( (s_A + n^2)  \frac{R_X^2 + R_S^2}{\epsilon_{\mathrm{abs}}} \right)$ classical operations and
    $$
    \widetilde{\mathcal{O}}\left(\left(m \kappa_A^2( 1 + R_y)^2 + n^2 ({\kappa_A^2}( 1 + R_y) + R_X)\right)\frac{(R_X+R_S)^3}{\epsilon_{\mathrm{abs}}^3}\right)
    $$
    quantum gates and queries.
\end{theorem}

\section{Numerical simulation}
\label{sec:numerical-simulation}
This section includes a small deterministic numerical experiment to illustrate the behavior of the full QADMM iteration with an operator-level simulation of the QSVT spectral update. Its purpose is to check whether the inexact QSVT-polynomial update preserves the convergence behavior of ADMM on a concrete SDP instance. 

Consider the Max-Cut SDP relaxation~\cite{goemans1995improved} on a weighted undirected graph,
\begin{equation}
    \min_{X\in\mathbb{S}^n}\ \langle C,X\rangle
    \quad \mathrm{s.t.}\quad
    \mathrm{diag}(X)=\mathbf{1},\quad X\succeq 0 ,
\end{equation}
where $C=-L/4$ and $L$ is the graph Laplacian. The graph has $n=8$ vertices. We first place a unit-weight cycle on the vertices and then add eight additional non-cycle edges. The extra edges are generated deterministically using the NumPy random generator; their weights are sampled independently from the uniform distribution on $[0.1,0.8]$. Thus the data matrix $C$ and the right-hand side $b=\mathbf{1}$ are fixed across all runs.

In the numerical experiment, Algorithm~\ref{alg:QADMM-for-SDP} is implemented in Python, with the $S$-update implemented by the operator-level QSVT simulation. At iteration $k$, we form
$V^{k+1}=C-\mathcal{A}^*(y^{k+1})-\gamma X^k$ with $\gamma=0.5$ and normalize it by
$
B_k=\|C\|_{\mathrm{F}}+\|y^{k+1}\|_1+\gamma\|X^k\|_{\mathrm{F}} .
$
This produces the Hermitian contraction $W_k=V^{k+1}/B_k$. The script constructs the exact one-ancilla block-encoding
$
U_{W_k}=\begin{pmatrix}
W_k & \sqrt{I-W_k^2}\\
\sqrt{I-W_k^2} & -W_k
\end{pmatrix},
$
and then applies the QSVT-compatible polynomial $P_d=b_d/2$, where $b_d$ is the Bernstein polynomial approximation of $x\mapsto\max\{x,0\}$ on $[-1,1]$. In the notation of Algorithm~\ref{alg:QADMM-for-SDP}, this operator-level simulation has $\tilde V^{k+1}=V^{k+1}$, $e_V^{k+1}=B_k$, and $v_V^{k+1}=W_k$. Hence the QSVT spectral transformation produces the intermediate update
$\tilde S^{k+1}=2B_kP_d(W_k)$,
which corresponds to~\eqref{eq:qadmm-tilde-S}. The classical iterate used in the residual and in subsequent updates is then
$S^{k+1}=\mathrm{Proj}_{\mathcal{S}}(\tilde S^{k+1})$,
where in this experiment the projection enforces the prescribed Frobenius-radius bound after the polynomial spectral update.
For reference, we also run the same ADMM iteration with the exact spectral projection $S^{k+1}=\mathrm{Proj}_{\mathbb{S}_+^n}(V^{k+1})$, which serves as the standard-projection baseline.
We test the fixed polynomial degrees $d\in\{512,2048,8192\}$. After the spectral update, both $S^{k+1}$ and $X^{k+1}$ are clipped to the Frobenius ball of radius $2n$, matching the bounded-iterate framework used in the convergence analysis. All methods start from $X^0=0$ and $S^0=0$ and run for $700$ iterations. 
Figure~\ref{fig:qsvt-qadmm-simulation} reports the dual constraint residual and objective gap during the iterations. The reference optimum $p^\star$ is computed by the CLARABEL~\cite{goulart2026clarabel} solver in CVXPY~\cite{diamond2016cvxpy}.

\begin{figure}[t]
    \centering
    \includegraphics[width=0.98\textwidth]{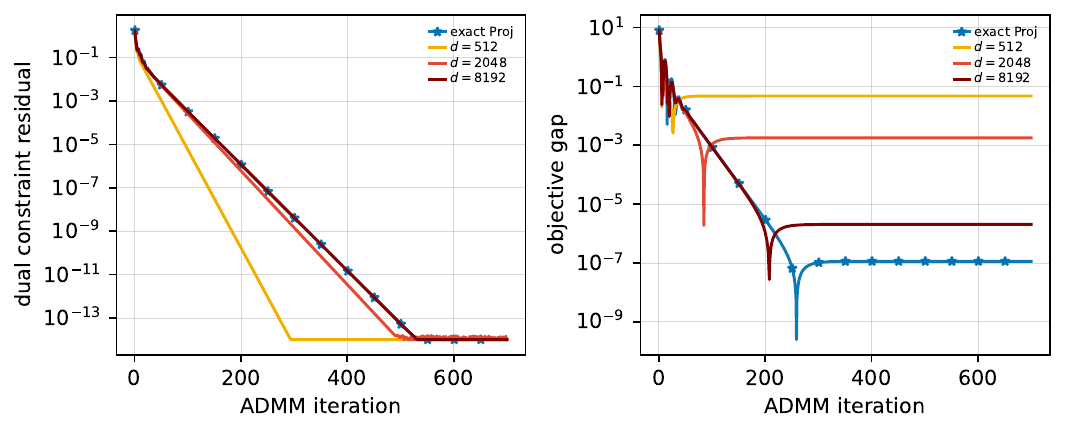}
\caption{Numerical simulation of QADMM. 
\texttt{exact Proj} denotes ADMM with the exact spectral projection $\mathrm{Proj}_{\mathbb{S}^n_+}$, and the other three curves denote QADMM runs with different QSVT polynomial degrees $d$. The left panel shows the dual constraint residual $\|\mathcal{A}^*(y^k)+S^k-C\|_{\mathrm{F}}$, and the right panel shows the objective gap $|b^\top y^k-p^\star|$ to the reference SDP optimum.}
    \label{fig:qsvt-qadmm-simulation}
\end{figure}

The simulation shows that the QSVT-polynomial updates track the convergence behavior of the standard spectral-projection ADMM on this instance. Increasing the fixed polynomial degree reduces the inexactness of the spectral update and gives objective values closer to the optimum; with $d=8192$, the final objective gap is approximately $2.0\times 10^{-6}$. A sharp dip followed by a rebound in the objective-gap curve occurs when the signed objective error $b^\top y^k-p^\star$ crosses zero early.

\section{Conclusion}

\label{sec: conclusion}
We present a QADMM algorithm for SDP based on an inexact ADMM framework, which leverages quantum spectral transformation to replace the most computationally intensive classical matrix eigendecomposition steps. By carefully designing a polynomial proximal operator, the iterative scheme admits an efficient implementation via QSVT. 
A detailed complexity analysis shows that QADMM improves upon both classical ADMM and quantum solvers in regimes of  dimension $n$ or accuracy $\epsilon$. The main limitation of our algorithm is that it requires polynomial-size QRAM data structures for encoding the data matrices and intermediate classical iterates, which is a strong assumption. Future work will focus on relaxing this assumption and exploring the potential of QADMM for other convex optimization problems such as convex conic programming and convex quadratic programming.

\section*{Author contributions}
Hantao Nie proposed the idea, wrote the manuscript, and carried out the numerical experiments. Dong An proposed the idea and contributed to the manuscript writing. Zaiwen Wen proposed the idea. All authors discussed the results and approved the final manuscript.
Large language models were used for language polishing and coding assistance. All authors reviewed the outputs and take full responsibility for the content of the manuscript.

\begin{acknowledgements}
Dong An acknowledges funding from Quantum Science and Technology - National Science and Technology Major Project via Project 2024ZD0301900, and the support by The Fundamental Research Funds for the Central Universities, Peking University. 
We thank Tamas Terlaky for the helpful discussion on the recent developments of quantum SDP algorithms.

\end{acknowledgements}

\appendix
\section{Literature Review}\label{app:sec: literature review}

\textbf{Classical Methods for SDP}

\noindent
The development of efficient classical algorithms for SDP has been a major focus in optimization research.  
Cutting plane methods are a class of algorithms that iteratively refine feasible regions by adding linear constraints (cuts) to the problem. These methods are particularly effective for large-scale SDPs, as they can exploit the sparsity of the underlying data. In~\cite{lee2015faster}, the authors proposed a cutting plane method for solving SDPs with a complexity of 
$$
{O}\left( m (m^2 + n^{\omega} + s_A)\right),
$$
where $\omega < 2.376$ is the exponent of matrix multiplication and $s_A$ is the number of non-zero entries in all the constraint matrices.

The matrix multiplicative weights update method (MMWU) is a powerful technique for solving game theory problems. MMWU is a primal-dual algorithm that iteratively updates the primal and dual variables based on the observed gradients of the objective function. It is shown that MWU is equivalent to the mirror descent method with 
quantum relative entropy as the Bregman divergence (Proposition 8.5 in~\cite{nannicini2024quantum}). 
When applied to SDPs, MMWU has a complexity of 
$$
    \mathcal{O}\left(m n^2 \cdot \operatorname{poly}\left(\frac{R_{\mathrm{Tr}X} R_{y}}{\epsilon_{\mathrm{abs}}}\right)\right) ,
$$
where $R_{\mathrm{Tr}X}$ is an upper bound on the trace of the primal solutions, $R_y$ is an upper bound on the $l_1$ norm of the dual solutions, and $\epsilon_{\mathrm{abs}}$ is the absolute accuracy of the objective function value.

Interior point methods for SDPs were pioneered by Karmarkar in the 1980s.
By solving a sequence of Newton systems derived from barrier function optimality conditions, IPMs
provide polynomial-time solvability for SDPs and are highly effective for small to medium problem sizes. 
In~\cite{jiang2020faster}, the authors proposed an interior point method (IPM) for solving semidefinite programs (SDPs) that requires
${O}(\sqrt{n} \log(1 / \epsilon))$ iterations,
with a per-iteration cost of $\mathcal{O}(m n^2 + m^\omega + n^\omega)$.
Consequently, the overall complexity is
$$
\widetilde{\mathcal{O}}\left(\sqrt{n}\left(m n^2+m^\omega+n^\omega\right) \right),
$$
where $\epsilon_{\mathrm{gap}}$ is the relative accuracy measured by the relative duality gap. In the literature of IPMs, an $\epsilon_{\mathrm{gap}}$-optimal solution measured by relative accuracy is defined as a solution that satisfies 
$$
\langle C, X\rangle \geq \langle C, X^* \rangle - \epsilon_{\mathrm{gap}} \| C \|
\text{ and }
\|\mathcal{A}(X) - b\|_1 \leq \epsilon_{\mathrm{gap}} (\|b\|_1 + \sum_{i=1}^{m}\|A_i\|) .
$$
Many prominent SDP solvers are based on IPMs and can reliably solve moderate-sized SDPs. However, interior point approaches become computationally expensive for large-scale SDPs, since each iteration requires factorizing large Hessian or Schur complement matrices. This per-iteration cost grows steeply (often cubic in the matrix dimension), which limits the scalability of IPMs on very high-dimensional problems.

To handle large problem instances where IPMs struggle, researchers have turned to operator splitting-based methods that use only gradient or operator splitting steps instead of second-order Hessian computations. These approaches include alternating direction augmented Lagrangian methods (SDPAD)~\cite{wen2010alternating}, ADMM-based interior point methods (ABIP)~\cite{deng2024enhanced}, operator splitting with homogeneous self-dual embedding (SCS)~\cite{o2016conic}. 
In exchange, they typically attain solutions of modest accuracy and may require many more iterations to converge. Operator splitting-based methods perform particularly well in numerical experiments, but few of them have been rigorously analyzed in terms of complexity.
ABIP~\cite{deng2024enhanced} is proved to converge in a number of iterations as
$
    \mathcal{O}\left(\frac{\kappa_A^2\|\mathbf{Q}\|^2}{\sqrt{\epsilon_{\mathrm{abs}}}} \log \left(\frac{1}{\epsilon_{\mathrm{abs}}}\right)\right),
$
where $\kappa_A$ is the condition number of the constraint matrix, $\mathbf{Q}$ is the coefficient matrix of the HSD formulation with $\|\mathbf{Q}\|^2 = \mathcal{O}(\sum_{i=1}^m \|A_i\|^2 + \|b\|^2 + \|C\|^2)$. 
Each iteration costs $\mathcal{O}(n^4)$ operations for matrix multiplication and requires one eigenvalue decomposition of size $n$.
If we assume that the eigenvalue decomposition has the same complexity as matrix multiplication, then the overall complexity of ABIP is 
$$
\widetilde{\mathcal{O}}\left(n^4\frac{\kappa_A^2\|\mathbf{Q}\|^2}{\sqrt{\epsilon_{\mathrm{abs}}}}\right) .
$$
We believe that the complexity of first order methods can be improved by using quantum computers.

\textbf{Quantum Algorithms for SDP}

Given the computational challenges of solving very large SDPs, a growing body of work explores quantum algorithms that could potentially offer polynomial speed-ups over classical methods. Quantum computers can, in principle, accelerate linear algebra subroutines (like solving linear systems or eigenvalue estimation) that are central to SDP solvers. The goal of quantum SDP solvers is typically to improve the dependence on the problem dimension ($n$) or number of constraints ($m$) by leveraging quantum linear algebra. 

Initial breakthroughs in quantum SDP solving were achieved by Brandão and Svore~\cite{brandao2017quantum}. 
The quantum matrix multiplicative weights update (QMWU) method they proposed is built on a combination of quantum Gibbs sampling and the multiplicative weight updates framework. QMWU draws inspiration from a classical heuristic of MWU for SDP, but replaces the costly inner linear programming steps with quantum subroutines. This quantum algorithm prepares quantum states that encode an approximate solution and iteratively updates them using a form of matrix exponential weights. 
Their method achieves a worst-case running time of 
$$
{\mathcal{O}}\left(\sqrt{mn} s^2 \mathrm{poly}(\log(n), \log(m), R_{\mathrm{Tr}X}, R_y, \frac{1}{\epsilon_{\mathrm{abs}}})\right)
$$
with $s$ the row-sparsity of the input matrices. 
Specifically, their algorithm provides a quadratic improvement over the best known classical runtime dependence on $n$ and $m$. 

In addition to the algorithm itself, Brandão and Svore proved that its speed-up was essentially optimal in broad generality: they established a quantum lower bound of $\Omega(\sqrt{n} + \sqrt{m})$ for solving SDPs when other parameters are held constant. This result indicates that no quantum algorithm can asymptotically outperform the $\mathcal{O}(\sqrt{nm})$ time scaling in full generality, barring improvements in logarithmic or problem-dependent factors.

Following this pioneering work, van Apeldoorn et al.~\cite{van2017quantum} developed improved quantum SDP solvers that refined the complexity dependence on various input parameters. Van Apeldoorn et al. introduced new quantum algorithmic techniques including a method to efficiently implement smooth functions of a sparse Hamiltonian and a generalized minimum-finding procedure, which together yielded better runtime scaling with respect to these parameters. Their improved quantum SDP solver essentially matches the favorable $\mathcal{O}(\sqrt{nm})$ scaling in the main problem sizes while bringing down the overhead polynomial factors in $s$, $1/\delta$. The worst case complexity of their algorithm is
$$
\widetilde{\mathcal{O}}\left(\sqrt{n m} s^2\left(\frac{R_{\mathrm{Tr}X} R_y}{\epsilon_{\mathrm{abs}}}\right)^8\right) .
$$
In the same work, they showed that for certain SDPs arising from combinatorial optimization with high symmetry, the quantum speed-up may disappear. More fundamentally, the authors proved a worst-case lower bound: for general linear programs (including SDPs), any quantum algorithm must take time linear in $m\times n$ when $m$ is on the order of $n$. In other words, in the dense, worst-case regime, a quantum solver cannot asymptotically outperform classical solvers. This result reinforced the understanding that quantum speed-ups for SDPs are achievable in certain regimes.

Further advances in quantum SDP algorithms were made by van Apeldoorn and Gilyén~\cite{van2018improvements}, who provided a unified framework that improved and generalized all prior quantum SDP solvers. They focused on constructing more efficient quantum Gibbs samplers for different models of input access, leading to better algorithmic bounds. For the standard sparse-matrix input model (where the algorithm can query entries of the constraint matrices), their quantum SDP solver runs in time 
$$
\widetilde O\left( (\sqrt{m} + \sqrt{n} \frac{R_{\mathrm{Tr}X} R_y}{\epsilon_{\mathrm{abs}}}) s \left(\frac{R_{\mathrm{Tr}X} R_y}{\epsilon_{\mathrm{abs}}}\right)^4 \right).
$$ 
Notably, they applied their improved algorithm to Aaronson's shadow tomography problem (estimating properties of an unknown quantum state) and obtained a more efficient solution than was previously known. They also showed quantum speed-ups for tasks like quantum state discrimination and $E$-optimal experimental design by formulating them as SDPs. 
In each case, the quantum SDP approach beat the best known classical complexity in certain parameters, albeit sometimes at the expense of worse dependence on other parameters. Finally, van Apeldoorn and Gilyén strengthened the theoretical foundations by proving new lower bounds for quantum SDP-solving. 
They showed that any quantum algorithm must incur at least a $\widetilde\Omega(\sqrt{m})$ factor (up to logarithmic terms) in its runtime complexity for SDPs, and that polynomial dependence on certain problem parameters (like condition numbers or error) is necessary in the quantum setting. 
These results confirmed that the scaling achieved by their algorithms is essentially tight and unavoidable without further assumptions. 

Recently, researchers have revisited the interior point method paradigm in the quantum context. Augustino et al.~\cite{augustino2023quantum} proposed the first quantum interior point methods (QIPMs) for SDP, aiming to combine the convergence guarantees of IPMs with quantum linear algebra speed-ups. Their work introduced two variants of a quantum interior point algorithm, both leveraging quantum linear system solvers to handle the Newton step at each iteration. 
The first variant II-QIPM closely mirrors a classical IPM but allows an inexact search direction solved by a quantum linear solver, which can lead to iterates that temporarily violate feasibility. This method achieves a runtime of
$$
\widetilde{\mathcal{O}}\left(\left(n^{5.5} \kappa_{newt} \kappa_{\mathcal{A}} \rho\left(\|\mathcal{A}\|_{\mathrm{F}}+\rho n^{1.5}\right)\right) \frac{1}{\epsilon_{\mathrm{gap}}}\right),
$$
where $\kappa_{\mathcal{A}}$ is the condition number of the matrix $\mathcal{A}$ whose columns are the vectorized constraint matrices $A^{(1)}, \ldots, A^{(m)}$, $\kappa_{newt}$ is the condition number of the Newton system, and $\rho>0$ denotes the size of the initial infeasible solution and is generally considered a constant in many papers. 

The second variant IP-QIPM uses a novel nullspace-based Newton system formulation to ensure that each step stays within the feasible region, even though the Newton equations are solved only approximately. 
This method solves the primal-dual SDO pair with $m=\mathcal{O}\left(n^2\right)$ constraints to $\epsilon_{\mathrm{gap}}$-optimality using at most
$$
\tilde{O}\left(n^{3.5} \frac{\kappa_{newt}^2}{\epsilon_{\mathrm{gap}}} \right) 
$$
QRAM accesses and $\mathcal{O}\left(n^{4.5}\right. \mathrm{poly}\log \left.\left(n, \kappa_{newt}, \frac{1}{\epsilon_{\mathrm{gap}}}\right)\right)$ arithmetic operations. 
Compared to classical IPMs, their algorithms run faster in the large-$n$ regime, but the cost grows more steeply with $\frac{1}{\epsilon_{\mathrm{gap}}}$ (accuracy requirements) and with certain spectral condition measures than in classical IPMs. This reflects a common theme in quantum optimization algorithms: one can often trade off some dependence on precision or problem “niceness” for gains in raw dimension or size scaling. 

In addition to the baseline QIPM results, Mohammadisiahroudi et al.~\cite{mohammadisiahroudi2025quantum} develop an iterative-refinement (IR) framework for semidefinite optimization that repeatedly calls any limited-precision interior-point-type oracle at a fixed accuracy and provably upgrades the solution to a target precision in refinement steps, thereby turning the overall dependence on accuracy from polynomial to logarithmic while preserving the best known dimension and conditioning scalings of the underlying oracle. 
Concretely, with the costs of 
$$
\widetilde{\mathcal{O}}\left(n^{3.5}\kappa_0^{2}\right)
$$
QRAM accesses and $\widetilde{\mathcal{O}}\left(n^{4.5}\right)$ arithmetic operations, the wrapped IR-QIPM inherits these per-call bounds and achieves quadratic convergence of the duality gap without relying on strict complementarity or nondegeneracy assumptions. The authors note a practical caveat that the cost per refinement step can increase with conditioning, but the worst-case accuracy dependence is exponentially improved relative to directly running a high-precision QIPM.

\section{Review of ADMM} \label{app:sec: review of ADMM}

\subsection{ADMM for convex optimization}

In general, consider a convex optimization problem of the form:
$$
\begin{aligned}
\min_{u, v} &\quad  f(u) + h(v)\quad
\text{s.t. } & Au + Bv = c~,
\end{aligned}
$$
where $f(u)$ and $h(v)$ are convex functions and $A u + B v = c$ is an equality constraint coupling the decision variables $u$ and $v$. The augmented Lagrangian function for this problem introduces a dual variable $\Lambda$ for the constraint and a penalty parameter $\gamma>0$ to enforce the constraint softly:
$$
\mathcal{L}_\gamma(u,v,\Lambda) = f(u) + h(v) + \langle \Lambda, Au + Bv - c \rangle + \frac{1}{2\gamma}\|A u + B v - c\|^2.
$$
The augmented Lagrangian function is a penalized version of the Lagrangian function with an additional quadratic term that penalizes the violation of the constraints. 
Compared to the Lagrangian function, 
the augmented Lagrangian function serves as a tighter lower bound on the objective function. By coupling the primal and dual variables in the augmented Lagrangian function, we transform the original constrained optimization problem into a saddle point problem 
$$
\min_{u, v } \max_{\Lambda} \mathcal{L}_\gamma(u,v,\Lambda) .
$$
An ADMM iteration consists of alternating optimization of this augmented Lagrangian function with respect to $u$ and $v$, updating the dual variable by one gradient step. The classical ADMM algorithm for the above problem is summarized as follows:

1. $u$-update: $u^{k+1} := \displaystyle\arg\min_{u} \mathcal{L}_\gamma(u, v^k, \Lambda^k)$. \\
2. $v$-update: $v^{k+1} := \displaystyle\arg\min_{v} \mathcal{L}_\gamma(u^{k+1},v , \Lambda^k)$. \\
3. Dual update: $\Lambda^{k+1} := \Lambda^k + \frac{1}{\gamma} (A u^{k+1} + B v^{k+1} - c)$.

\subsection{Classical ADMM for SDP} \label{app:sec: classical ADMM for SDP}

Assume $(X^*, y^*, S^*)$ satisfies the optimality conditions:
\begin{align}
    \mathcal{A}(X^*) &= b,\quad X^* \succeq 0,      \tag{Primal feasibility}\\
    \mathcal{A}^*(y^*) + S^* &= C,\quad S^* \succeq 0, \tag{Dual feasibility}\\
    X^* S^* &= 0.                                  \tag{Complementary slackness}
\end{align}
Starting from an initial guess $(X_0,y_0,S_0)$,
classical ADMM iteratively updates the primal and dual variables as follows:
\begin{subequations}
    \begin{align}
         y^{k+1} & :=-\left(\mathcal{A A}^*\right)^{-1}(\gamma(\mathcal{A}(X^k)-b)+\mathcal{A}(S^k-C)), \label{app:eq: y in classical ADMM for SDP} \\
         S^{k+1} & :=\mathrm{Proj}_{\mathbb{S}^n}\left(C-\mathcal{A}^*(y^{k+1})-\gamma X^k \right), \label{app:eq: S in classical ADMM for SDP}\\
         X^{k+1} & := X^k+\frac{1}{\gamma}(\mathcal{A}^*\left(y^{k+1}\right)+S^{k+1}-C)
         .
         \label{app:eq: X in classical ADMM for SDP} 
    \end{align}
\end{subequations}
We denote $f_1(X) = \langle C, X\rangle + \delta_{\{\mathcal{A}(X)=b\}}(X)$ and compute its proximal operator as
\begin{equation}
    \label{app:eq: proximal operator of f_1}
    \mathrm{prox}_{f_1 / \gamma}(Z) = Z - \frac{C}{\gamma} - \mathcal{A}^{*}(\mathcal{A}\mathcal{A}^*)^{-1}(\mathcal{A}(Z - \frac{C}{\gamma}) - b).
\end{equation}
By introducing a variable 
$$
Z^{k+1} = X^{k}+\frac{1}{\gamma}(\mathcal{A}^*\left(y^{k+1}\right)-C),$$
the iterate scheme~\eqref{app:eq: y in classical ADMM for SDP}--\eqref{app:eq: X in classical ADMM for SDP} is rewritten as a fixed point iteration
$
Z^{k+1} = \mathcal{T}(Z^k) 
$,
where $\mathcal{T}$ is a nonexpansive operator defined as
\begin{equation}
    \label{app:eq: fixed point mapping}
\mathcal{T}(Z^k) := Z^k + \mathrm{prox}_{f_1 / \gamma}(2 \mathrm{Proj}_{\mathbb{S}_+^n}(Z^k) - Z^k) - \mathrm{Proj}_{\mathbb{S}_+^n}(Z^k).
\end{equation}
This equivalent fixed point iteration is called Douglas-Rachford splitting (DRS)~\cite{douglas1956numerical}.
Moreover, the sequence $\{(X^k, y^k, S^k)\}$ is recovered from $Z^k$ by
\begin{equation}
\begin{aligned}
X^{k} &= \mathrm{Proj}_{\mathbb{S}^n}(Z^k), \\
S^{k} &= \mathrm{Proj}_{\mathbb{S}^n}( - \gamma Z^k) , \\
y^{k+1} &= -\left(\mathcal{A A}^*\right)^{-1}(\gamma(\mathcal{A}(X^k)-b)+\mathcal{A}(S^k-C)).
\end{aligned}
\end{equation}

\subsection{Complexity of classical ADMM for SDP} \label{app:sec: complexity of classical ADMM for SDP}
The following proposition shows the convergence of the ADMM algorithm for SDP.
This is a direct application of the classical ADMM convergence theory (Theorem 3.2, 3.3 in~\cite{lin2022alternating}). 

\begin{proposition}
    \label{app:thm: convergence of exact ADMM}
Assume $(X^*, y^*, S^*)$ satisfies the optimality conditions.
Let $D=\gamma \|X^0-X^*\|^2+\frac{1}{\gamma}\|S^0-S^*\|^2$.

\noindent
1. (asymptotic convergence) The sequence $\{X^k, y^k, S^k\}$ generated by the scheme~\eqref{app:eq: y in classical ADMM for SDP}--\eqref{app:eq: X in classical ADMM for SDP}
 satisfies
$$
- b^\top y^{k} + b^\top y^* 
\to 0, \quad \mathcal{A}^*(y^{k}) + S^{k} - C \to 0, \quad \text{as} \quad k \to \infty.
$$

\noindent
2. 
(non-ergodic convergence) The last iterate $\{X^K, y^K, S^K\}$ generated by the scheme~\eqref{app:eq: y in classical ADMM for SDP}--\eqref{app:eq: X in classical ADMM for SDP}
 satisfies
$$
\begin{aligned}
 - b^\top y^K  + b^\top y^*
&\leq \frac{D}{K}+\frac{2 D}{\sqrt{K}}+\left\|X^*\right\| \sqrt{\frac{\gamma D}{ K}}, \quad 
\left\|\mathcal{A}^*(y^K) + S^K - C\right\|  \leq \sqrt{\frac{\gamma D}{K}}.
\end{aligned}
$$

\noindent
3. (ergodic convergence) The average iterate $\bar X^K = \frac{1}{K} \sum_{k=1}^K X^k$, $\bar y^K = \frac{1}{K} \sum_{k=1}^K y^k$, and $\bar S^K = \frac{1}{K} \sum_{k=1}^K S^k$  generated by the scheme~\eqref{app:eq: y in classical ADMM for SDP}--\eqref{app:eq: X in classical ADMM for SDP} satisfies
$$
\begin{aligned}
    -b^\top \bar{y}^K + b^\top y^* 
    &
    \leq \frac{D}{2K}+\frac{2 \sqrt{\gamma D}\left\|X^*\right\|}{K}, \quad 
    \left\|\mathcal{A}^*(\bar{y}^K) + \bar{S}^K - C\right\| \leq \frac{2 \sqrt{\gamma D}}{K} .
\end{aligned}
$$
\end{proposition}

A direct consequence of the above proposition is the following iteration complexity of the classical ADMM algorithm for SDP.
\begin{corollary}
    To achieve an $\epsilon_{\mathrm{abs}}$-accuracy solution,
    the classical ADMM algorithm with initialization $(X^0, S^0)$ = $(0, 0)$ requires 
    $$
    K = \mathcal{O}\left(\frac{R_X^2 + R_S^2}{\epsilon_{\mathrm{abs}}}\right) 
    $$
    iterations, where $R_X, R_S$ are the upper bounds of $\| X^*\|$ and $\| S^* \|$, respectively.
\end{corollary}

\begin{theorem}
    
To obtain an $\epsilon_{\mathrm{abs}}$-optimal solution that satisfies 
\begin{equation}
   \label{app:eq: dual epsilon optimal} 
- b^\top \bar y^K + b^\top y^* \leq \epsilon_{\mathrm{abs}}, \quad \left\|\mathcal{A}^*(\bar y^K) + \bar S^K - C\right\| \leq \epsilon_{\mathrm{abs}},
\end{equation}
the classical ADMM requires a running time of  
$$
 O\left(m^3 + m^2 n^\omega + (m^2 + s_A + n^\omega)\cdot \frac{R_X^2 + R_S^2}{\epsilon_{\mathrm{abs}}}\right), 
$$
where $R_X$ and $R_S$ are the upper bounds on the primal and dual optimal solutions $X^*, S^*$ respectively, $s_A$ is the number of non-zero entries in all the constraint matrices, and $\omega$ is the exponent of matrix multiplication.
\end{theorem}
\begin{proof}
    Assume we compute $\mathcal{A} \mathcal{A}^*$ and
prepare a Cholesky factorization for it, which has a computational cost of $\mathcal{O}(m^2 n^\omega + m^3)$. 
Now we consider the complexity of each iteration. 
Solving the linear system $\mathcal{A A}^* y = U$ requires $\mathcal{O}(m^2)$ time. The eigenvalue decomposition of size $n$ has the same complexity as matrix multiplication, i.e., $\mathcal{O}(n^\omega)$, where $\omega$ is the exponent of matrix multiplication. Hence the projection onto $\mathbb{S}_+^n$ requires $\mathcal{O}(n^\omega)$ time.

Assume $\mathbf{A}$ has $s_A$ non-zero entries. 
As shown in Appendix~\ref{app:sec: linear transformation}, the costs of computing $\mathcal{A}(X)$ and $\mathcal{A}^*(y)$ are both $\mathcal{O}(s_A)$.
Therefore,
step~\eqref{app:eq: y in classical ADMM for SDP} requires $\mathcal{O}(m^2 + s_A + n^2)$ time, step~\eqref{app:eq: S in classical ADMM for SDP} requires $\mathcal{O}(n^\omega + s_A + n^2)$ time, and step~\eqref{app:eq: X in classical ADMM for SDP} requires $\mathcal{O}(s_A + n^2)$ time. Therefore the total complexity of each iteration is $\mathcal{O}(m^2 + s_A + n^\omega)$.
Combining the iteration complexity and the iteration cost, we have that the total complexity of the classical ADMM algorithm is
$$
 O\left(m^3 + m^2 n^\omega + (m^2 + s_A + n^\omega) \cdot \frac{R_X^2 + R_S^2}{\epsilon_{\mathrm{abs}}}\right).
$$
This completes the proof.
\end{proof}

\section{Implementation details for QADMM}
\label{app:sec: implementation details for QADMM}
\subsection{Classical implementation of linear mapping} 
\label{app:sec: linear transformation}

The $\mathrm{svec}$ operator transforms an $n \times n$ symmetric matrix $X$ to a vector $\mathrm{svec}(X) \in \mathbb{R}^{n(n+1)/2}$, which contains the upper triangular part of $X$ excluding the diagonal. The $\mathrm{smat}$ operator transforms a vector $\mathrm{svec}(X)$ back to a symmetric matrix $X \in \mathbb{S}^n$.
More  specifically, the $\mathrm{svec}$ and $\mathrm{smat}$ operators are defined as follows:
\begin{subequations}
    \begin{align}
    & \mathrm{svec}\left(
        X \right) = \left( X_{11}, X_{12}, \ldots, X_{1n}, X_{22}, \ldots, X_{2n}, \ldots, X_{nn} \right)^{\top}, \label{app:eq: svec} \\
    & \mathrm{smat}\left( (X_{11}, X_{12}, \ldots, X_{1n}, X_{22}, \ldots, X_{2n}, \ldots, X_{nn} )^\top\right) = X
    \label{app:eq: smat}.
    \end{align}
\end{subequations}

On classical computers, the linear operator 
$$
\mathcal{A}: \mathbb{S}^n \to \mathbb{R}^m, \quad \mathcal{A}(X) = \left( \langle A^{(1)}, X \rangle, \langle A^{(2)}, X \rangle, \ldots, \langle A^{(m)}, X \rangle \right)^\top
$$
can be represented as a matrix $\boldsymbol{A} \in \mathbb{R}^{m \times n(n+1)/2}$, where the $i$-th row of $\boldsymbol{A}$ is $\mathrm{svec}(A^{(i)})$. 
Given a symmetric matrix $X \in \mathbb{S}^n$ and a vector $y \in \mathbb{R}^m$, the linear operator $\mathcal{A}$ and its adjoint operator $\mathcal{A}^*$ are computed as follows:
\begin{equation}
    \mathcal{A}(X) = A \mathrm{svec}(X), \quad \mathcal{A}^*(y) = \mathrm{smat}(\boldsymbol{A}^{\top} y).
\end{equation}

Assume $\boldsymbol{A} \in \mathbb{R}^{m \times n(n+1)/2}$ is a sparse matrix with $s_A$ non-zero entries. 

\textbf{Complexity of computing $\mathcal{A}(X)$}

\noindent
The cost of computing $\mathcal{A}(X)$ is the sum of the vectorization cost of $\mathrm{svec}(X)$ and the matrix-vector multiplication cost of $\boldsymbol{A}(X)$. The vectorization cost is $\mathcal{O}(n^2)$, and the matrix-vector multiplication cost is $\mathcal{O}(s_A)$. Therefore, the total cost of computing $\mathcal{A}(X)$ is $T_{A} = \mathcal{O}(s_A + n^2)$. As a special case, if $\boldsymbol{A}$ is a dense matrix, then the total cost of computing $\mathcal{A}(X)$ is $\mathcal{O}(mn^2 + n^2) = \mathcal{O}(mn^2)$.

\textbf{Complexity of computing $\mathcal{A}^*(y)$}

\noindent
The cost of computing $\mathcal{A}^*(y)$ is the sum of the matrix-vector multiplication cost of $\boldsymbol{A}^{\top} y$ and the matrix reconstruction cost of $\mathrm{smat}(\boldsymbol{A}^{\top} y)$. The matrix-vector multiplication cost is $\mathcal{O}(s_A)$, and the matrix reconstruction cost is $\mathcal{O}(n^2)$. Therefore, the total cost of computing $\mathcal{A}^*(y)$ is $T_{A^*} = \mathcal{O}(s_A + n^2)$. As a special case, if $\boldsymbol{A}$ is a dense matrix, then the total cost of computing $\mathcal{A}^*(y)$ is $\mathcal{O}(mn^2 + n^2) = \mathcal{O}(mn^2)$.

\subsection{Linear system in y-update}
Once $u^{k+1}$ has been computed classically and the normalized quantum state $| \hat u^{k+1}\rangle$ prepared, we then solve the linear system
$$
(\mathcal{A}\mathcal{A}^*) \frac{\hat y^{k+1}}{\|u^{k+1}\|} =  -|\hat u^{k+1}\rangle \quad \text{  or  } \quad (\mathbf{A}\mathbf{A}^\top) \frac{\hat y^{k+1}}{\|u^{k+1}\|} =  -|\hat u^{k+1}\rangle.
$$
However, the block-encoding of $\mathcal{A}\mathcal{A}^*$ is not available directly.
The preparation of $|\hat u^{k+1}\rangle$ uses the classical vector $u^{k+1}$ already formed in the $y$-update. Loading it into the state-preparation data structure costs $\mathcal{O}(m)$ classical writes per iteration, and this cost is included in the classical part of the complexity bound.
By introducing an auxiliary variable $t \in \mathbb{R}^{n^2}$, the linear system is equivalent to 
$$
\begin{bmatrix}
    0 & \mathbf{A} \\
    \mathbf{A}^\top & - \mathrm{I}_{n^2}
\end{bmatrix}
\begin{bmatrix}
    \frac{\hat y^{k+1}}{\|u^{k+1}\|} \\
    t
\end{bmatrix}
=
\begin{bmatrix}
    -|\hat u^{k+1}\rangle \\
    0
\end{bmatrix}.
$$
By Lemma 50 in~\cite{gilyen2019quantum}, if $\mathbf{A}$ is stored in QRAM, then a $(\|\mathbf{M}\|_{\mathrm{F}}, \mathcal{O}(\log(n)), \epsilon)$-block-encoding of the matrix $\mathbf{M} := \begin{bmatrix}
    0 & \mathbf{A} \\
    \mathbf{A}^\top &  - \mathrm{I}_{n^2}
\end{bmatrix}$ can be constructed in time $\tilde{\mathcal{O}}_{\frac{n}{\epsilon}}(1)$. Since this coefficient is independent of the number of iterations, we can construct the block-encoding of $\mathbf{M}$ once and use it in each iteration.
After that, we employ QLS to solve the linear system and perform quantum tomography to extract the information of $\hat{y}^{k+1}$.

\subsection{LCU in V-update}

Let the 
$(\alpha_i, \lceil \log(n) \rceil + 2, \xi)$-block-encodings of $A^{(1)}, \ldots, A^{(m)}, C, X^{k}$ be denoted by $U_{A^{(i)}}, U_C, U_{X^{k}}$, where $\alpha_i = \|A^{(i)}\|_{\mathrm{F}}$ and $\alpha_{m+1} = \|C\|_{\mathrm{F}}$, $\alpha_{m+2} = \|X^k\|_{\mathrm{F}}$.
Our goal is to implement an LCU circuit to obtain the matrix
\begin{equation}
\label{app:eq: lcu}
\sum_{i=1}^m \alpha_i y^{k+1}_i U_{A^{(i)}}  + 
\alpha_{m+1} U_{C} - \gamma \alpha_{m+2}  U_{X^k}.
\end{equation}

Let $r = \lceil \log(m) \rceil + 2$ be the number of ancilla qubits. Define
$
e_V^{k+1}=\sum_{i=1}^{m}\alpha_i |y_i^{k+1}|+\alpha_{m+1}+\gamma\alpha_{m+2}
$
and $\sigma_i^{k+1}=1$ if $y_i^{k+1}\geq 0$ and $\sigma_i^{k+1}=-1$ otherwise. Define preparation operator $O_{\rm prep}$ and select operators $O_{\rm sel}^{(1)}$, $O_{\rm sel}^{(2)}$ as follows: 
\begin{equation}
    \begin{aligned}
O_{\text {prep }}&:|0\rangle \mapsto \frac{1}{\sqrt{e_V^{k+1}}} \left(\sum_{j=0}^{m-1} \sqrt{\alpha_{j+1} |y_{j+1}^{k+1}|}|j\rangle + \sqrt{\alpha_{m+1} }|m\rangle + \sqrt{\gamma \alpha_{m+2}}|m+1\rangle \right), \\
O_{\text {sel }}^{(1)}&=\sum_{j=0}^{m-1}|j\rangle\langle j| \otimes \sigma_{j+1}^{k+1}U_{A^{(j+1)}} + |m\rangle\langle m| \otimes U_C + |m+1\rangle\langle m+1| \otimes \mathrm{Id}, \\
O_{\text {sel }}^{(2)}&= \sum_{j=0}^{m-1}|j\rangle\langle j| \otimes \mathrm{Id} + |m\rangle\langle m| \otimes \mathrm{Id} - 
|m+1\rangle\langle m+1| \otimes U_{X^k} .
    \end{aligned}
\end{equation}
A demonstration of the LCU is given below.

\begin{quantikz}[row sep=1cm, column sep=0.8cm]
  \lstick{$\ket{0}^{\otimes r}$}
    & \gate{O_{\rm prep}}
    & \gate[wires=2]{O_{\rm sel}^{(1)}}
    & \gate[wires=2]{O_{\rm sel}^{(2)}}
    & \gate{O_{\rm prep}^\dagger}
    & \rstick{ancilla}
  \\
  \lstick{$\ket{\psi}$}
    & \qw
    & \qw
    & \qw
    & \qw
    & \rstick{system}
\end{quantikz}

The preparation operator $O_{\rm prep}$ can be efficiently constructed using QRAM. 
Since $A^{(i)}$ and $C$ are fixed matrices, we can construct $O_{sel}^{(1)}$ by one query to the controlled version of the block-encodings of $A^{(i)}$ and $C$.

\section{Polynomial proximal operator}

\begin{lemma}
    \label{app:lem: polynomial approximation}
    1. For any $\epsilon > 0$, there exists a monotone increasing polynomial $g(x)$ on $[-1, 1]$ of degree $d = \mathcal{O}(\frac{1}{\epsilon})$ such that $g(-1)= 0$ and
    \begin{equation}
        \label{app:eq: polynomial approximation}
         | \max(0, x) - g(x) | \leq \epsilon, \quad \forall x \in [-1, 1].
    \end{equation}
    
    2. Let $\boldsymbol{g}$ be the corresponding spectral operator of $g$.
    Then $\boldsymbol{g}$ is a polynomial eigenvalue transformation on $\mathbb{S}^n$ and satisfies
    \begin{equation}
        \label{app:eq: polynomial approximation for SDP}
         \| \mathrm{Proj}_{\mathbb{S}_+^n}(X) - \boldsymbol{g}(X) \|_2 \leq \epsilon, \quad \forall X \in 
         \{ X \in \mathbb{S}^n: \| X \|_2 \leq 1\}  .
    \end{equation}
    \end{lemma}
\begin{proof}
    1.
    Denote $g_0 = \max(x, 0), x \in [-1 ,1]$. A well-known result in approximation theory states that a piecewise monotonic function $g$ can be approximated by a polynomial function with the same monotonicity~\cite{passow1974copositive, passow1974monotone, newman1979efficient}. By Lemma in\cite{passow1974copositive} there exists a sequence of polynomials $\{p_n\}_{n=1}^\infty$ such that
    \begin{equation}
        \|g_0 - p_n\|_\infty \leq C_2 M \frac{1}{n} ,
    \end{equation}
    where $C_2$ is an absolute constant, $\omega(g_0; \frac{1}{n})$ is the modulus of continuity of $g_0$ on the interval $[-1, 1]$, and $M = 1$ is the Lipschitz constant of $g_0$. 
    The construction of the polynomial $p_n$ is discussed in Section 3 of~\cite{passow1974copositive}. 
    This is equivalent to the following statement: for any $\epsilon' > 0$, there exists a monotone increasing polynomial $p_n$ of degree $n = \mathcal{O}(\frac{1}{\epsilon'})$ such that
    \begin{equation}
        \|g_0 - p_n\|_\infty \leq \epsilon' .
    \end{equation}
    Our goal is to find a polynomial such that it also satisfies $g(-1)=0$ and $g(x) \geq 0$ on $[-1,1]$. This can be achieved by the following transformation:
    $$
    g(x) = p_n(x) - p_n(-1) .
    $$
    Since $p_n$ is monotone increasing, $g$ is monotone increasing and $g(x)\geq g(-1)=0$ on $[-1,1]$. Take $\epsilon' = \frac{1}{2} \epsilon$; then, using $g_0(-1)=0$,
    $$
    | g(x) - g_0(x) | \leq |p_n(x)-g_0(x)| + |p_n(-1)-g_0(-1)| \leq 2 \epsilon' = \epsilon.
    $$ 
    
    \noindent
    2.
    By eigenvalue decomposition as $X = U \mathrm{diag}(\lambda_1, \ldots, \lambda_n) U^\top$, we have
    $$
    \| \mathrm{Proj_{\mathbb{S}_+^n}}(X) - \boldsymbol{g}(X) \|_2 = 
    \max_{1 \leq i\leq n} | g_0(\lambda_i) - g(\lambda_i) | \leq \epsilon .
    $$
    This completes the proof.
\end{proof}

Strict monotonicity in the following lemma is used only for the variational interpretation. If the polynomial approximation above is merely monotone on $[-1,1]$, we may add an arbitrarily small strictly increasing perturbation and extend it smoothly to $\mathbb R$; the additional approximation error is absorbed into $\delta_{\tilde S}$ and $\delta_{\tilde X}$.

\begin{lemma}
    \label{app:lem: proximal operator construction}
    Let $g: \mathbb{R} \to \mathbb{R}$ be a smooth function satisfying $g'(x)>0$ for all $x \in \mathbb{R}$. Denote $g^{-1}$ as the inverse function of $g$.
    Then 
    \begin{equation}
        \label{app:eq: proximal operator construction for scalar}
    h(x) = 
    \begin{cases}
        \frac{1}{\gamma} \int_{0}^{x} (g^{-1}(z) - z) dz , & \text{if } x \geq 0, \\
        \infty, & \text{if } x < 0,    
    \end{cases}
    \end{equation}
    satisfies $(1 + \gamma \partial h(x))^{-1} = g(x)$.
    
    \noindent
    2. 
    For $X \in \mathbb{S}^n$ with eigenvalue decomposition $X = U \mathrm{diag}(\lambda_1, \ldots, \lambda_n) U^\top$, where $U$ is a unitary matrix and $\lambda_i$ are the eigenvalues of $X$.
    Let $\boldsymbol{g}(X) = U \mathrm{diag}(g(\lambda_1), \ldots, g(\lambda_n)) U^\top$ be the eigenvalue transformation with respect to $g$.
    Define $\boldsymbol{h}: \mathbb{S}^n \to \mathbb{R}$ as
    $\boldsymbol{h}(X) = \sum_{i=1}^{n} h(\lambda_i)$, where $h$ is defined in~\eqref{app:eq: proximal operator construction for scalar}.
    Then $\boldsymbol{h}$ satisfies $(\mathrm{Id}+\gamma\partial \boldsymbol{h})^{-1}(S)=\boldsymbol{g}(S)$.
\end{lemma}

\begin{proof}
    \noindent
    1.
    The continuity and monotonicity of $g$ implies that $g^{-1}$ is also continuous and strictly monotonically increasing. From the definition of the proximal operator, we have
    \begin{equation}
         g(x) = \arg \min_{u \in \mathbb{R}} \left\{\frac{1}{2}\|x - u\|^2 + \gamma h(u)\right\} .
    \end{equation}
    This is equivalent to
    $ 0 \in g(x) - x + \gamma \partial h(g(x)) $ or,
    after replacing $g(x)$ by a generic argument $x$,
    $  \frac{1}{\gamma} \left(g^{-1}(x) - x\right) \in \partial h(x)$.
    Therefore the design of $h$ satisfies
    $$
    (1+\gamma\partial h)^{-1}(x)=g(x) .
    $$
    
    \noindent
    2.
    By eigenvalue decomposition as $X = U \mathrm{diag}(\lambda_1, \ldots, \lambda_n) U^\top$, 
    \begin{equation}
        \begin{aligned}
             \boldsymbol{g}(X) = \arg \min_{Y \in \mathbb{S}^n} \left\{\frac{1}{2}\|X - Y\|_{\mathrm{F}}^2 + \gamma \boldsymbol{h}(Y)\right\} 
        \end{aligned}
    \end{equation}
    is equivalent to the following optimization problem:
    $$
            \mathrm{diag}(g(\lambda_1), \ldots, g(\lambda_n)) = \arg \min_{Y'\in \mathbb{S}^n} \left\{\frac{1}{2}\| U \mathrm{diag}(\lambda_1, \ldots, \lambda_n) U^\top - U Y' U^\top\|_{\mathrm{F}}^2 + \gamma \sum_{i=1}^{n} h(\mu_i(Y'))\right\} ,
    $$
    where $\mu_i(Y')$ are the eigenvalues of $Y'$. Since both the Frobenius norm and $\boldsymbol h$ are spectral functions, the minimizer shares the eigenvectors of $X$ and the problem reduces to the scalar problems. The minimum is achieved if and only if  $ g(\lambda_i) = \arg \min_{u \in \mathbb{R}} \left\{\frac{1}{2}\|\lambda_i - u\|^2 + \gamma h(u)\right\}, i = 1, \ldots, n $.
    Combining the above equations with the result in part 1, we have the desired result.
    \end{proof}

\section{Complexity analysis}
\label{sec: complexity of QADMM for SDP}
\subsection{Iteration number of QADMM for SDP}

QADMM algorithm for solving
\begin{equation}
    \label{app:eq: dual SDP 2}
    \min_{y\in\mathcal{Y},S \in \mathcal{S}}  -b^\top y  \qquad \text{s.t.}\quad A^*(y) + S = C~,
    \end{equation}
iterates in the following scheme:
\begin{subequations}
    \begin{align}
        \hat y^{k+1} &=-\left(\mathcal{A A}^*\right)^{-1}(\gamma(\mathcal{A}(X^k)-b)+\mathcal{A}(S^k-C)), \label{app:eq: hat y in QADMM for SDP} \\
         \tilde y^{k+1} & := \hat y^{k+1} + \Delta \tilde y^{k+1} \label{app:eq: tilde y in QADMM for SDP}, \quad  \|\Delta \tilde y^{k+1}\|_2 \leq \delta_{\tilde y}, \\
         y^{k+1} & := \mathrm{Proj}_{\mathcal{Y}}(\tilde y^{k+1} + \Delta y^{k+1}) \label{app:eq: y in QADMM for SDP}, \quad  \|\Delta y^{k+1}\|_2 \leq \delta_y, \\
         \hat V^{k+1} &:= C-\mathcal{A}^*(\tilde y^{k+1})-\gamma X^k, \label{app:eq: hat V in QADMM for SDP} \\
         \tilde V^{k+1} & := \hat V^{k+1} + \Delta V^{k+1} \label{app:eq: tilde V in QADMM for SDP} , \quad  \|\Delta V^{k+1}\|_{\mathrm{F}} \leq \delta_V, \\
         \tilde S^{k+1} & := \boldsymbol{g}_1(\tilde V^{k+1}) ,\label{app:eq: tilde S in QADMM for SDP} \\
        S^{k+1} & := \mathrm{Proj}_{\mathcal{S}} (\tilde S^{k+1} + \Delta S^{k+1}) \label{app:eq: S in QADMM for SDP}, \quad  \|\Delta S^{k+1}\|_{\mathrm{F}} \leq \delta_S, \\
        \tilde X^{k+1} & := -\frac{1}{\gamma}\boldsymbol{g}_2(\tilde V^{k+1}) ,\label{app:eq: tilde X in QADMM for SDP} \\
        X^{k+1} & := \mathrm{Proj}_{\mathcal{X}}(\tilde X^{k+1} + \Delta X^{k+1}) \label{app:eq: X in QADMM for SDP}, \quad  \|\Delta X^{k+1}\|_{\mathrm{F}} \leq \delta_X.
    \end{align}
\end{subequations}
The variables with tilde notations $\tilde y^{k+1}, \tilde{S}^{k+1}, \tilde{X}^{k+1}$ stand for quantum representations and variables without any notation $y^{k+1}, S^{k+1}, X^{k+1}$ stand for classical representations. Additionally, variables with hat notations $\hat y^{k+1}, \hat V^{k+1}$ are intermediate variables.
The error terms $\Delta \tilde y^{k+1}, \Delta y^{k+1}, \Delta V^{k+1}, \Delta S^{k+1}, \Delta X^{k+1}$ in~\eqref{app:eq: hat y in QADMM for SDP}-\eqref{app:eq: X in QADMM for SDP}
are used to account for the inexactness, which are bounded by $\delta_{\tilde y}, \delta_y, \delta_V, \delta_S, \delta_X$, respectively.
Let 
\begin{subequations}
    \begin{align}
\hat S^{k+1} &= \mathrm{Proj}_{\mathbb{S}_+^n}(\tilde V^{k+1}), \label{app:eq: hat S in QADMM for SDP} \\
\Delta \tilde{S}^{k+1} &= \tilde{S}^{k+1} - \hat{S}^{k+1}, \\
\hat X^{k+1} &= -\frac{1}{\gamma}(\mathrm{Id} - \mathrm{Proj}_{\mathbb{S}_+^n})(\tilde V^{k+1}) = \frac{1}{\gamma}( \hat S^{k+1} - \tilde V^{k+1}), \label{app:eq: hat X in QADMM for SDP} \\
\Delta \tilde{X}^{k+1} &= \tilde{X}^{k+1} - \hat{X}^{k+1}, 
\end{align}
\end{subequations}
be the error term between the polynomial QSVT and projection operator. 
By Lemma~\ref{app:lem: polynomial approximation}, there exists a polynomial $\boldsymbol{g}_1$ with degree $\mathcal{O}(\frac{B}{\epsilon})$ such that $\| \mathrm{Proj}_{\mathbb{S}_+^n}(V) - \boldsymbol{g}_1(V) \|_2 \leq \epsilon$ for any $V\in \{ V \in \mathbb{S}^n \mid \| V \|_2 \leq B\}$. 
Similar properties hold for $\boldsymbol{g}_2$. Assume $\tilde{V}^{k+1}$ has a uniform bound $B$. After we choose proper polynomials $\boldsymbol{g}_1$ and $\boldsymbol{g}_2$ with degree $\mathcal{O}(\frac{B}{\delta_{\tilde S}})$, $\mathcal{O}(\frac{B}{\delta_{\tilde X}})$, respectively, it holds that
\begin{equation}
    \begin{aligned}
        \| \Delta \tilde{S}^{k+1} \|_2 & \leq \delta_{\tilde S}, \\
        \| \Delta \tilde{X}^{k+1} \|_2 & \leq \delta_{\tilde X}. 
    \end{aligned}
\end{equation}

\begin{theorem}
    \textbf{(ergodic convergence)}
    \label{app:thm: ergodic convergence of QADMM for SDP}
    Assume $(X^*, y^*, S^*)$ satisfies the optimality conditions.
    The average iterate $(\bar{X}^K, \bar{y}^K, \bar{S}^K) := \frac{1}{K} \sum_{k=0}^{K-1} (\tilde X^k, \hat y^k, \tilde S^k)$  
     generated by QADMM satisfies
    \begin{subequations}
    \begin{align}
    f\left(\bar{y}^K\right)-f\left(y^*\right) & \leq \mathcal{O}\left(\frac{\left\|X^0 - X^*\right\|_{\mathrm{F}}^2+\left\| S^0-S^*\right\|_{\mathrm{F}}^2}{K} +\delta \right)\label{app:eq: ergodic convergence of QADMM for SDP 1}, \\
    \left\|\mathcal{A}^*( \bar{y}^K ) + \bar{S}^K-C\right\|_{\mathrm{F}} & \leq \mathcal{O}\left(\frac{\left\|X^0-X^*\right\|_{\mathrm{F}}+\left\|S^0-S^*\right\|_{\mathrm{F}}}{K}+\delta\right) , \label{app:eq: ergodic convergence of QADMM for SDP 2} 
    \end{align} 
    \end{subequations}
     where $\delta = \Theta(\delta_{\tilde y} + \delta_y + \delta_V +\delta_X + \delta_S + \delta_{\tilde S} + \delta_{\tilde X})$ is the accumulated ADMM inexactness term. The QLS accuracy parameter $\delta_{\hat y}$ is chosen so that the induced iterate error satisfies $\delta_{\tilde y}=\Theta(\delta_{\hat y})$.
\end{theorem}

\begin{proof} For notational convenience, denote $z^k := \left(\tilde X^k, \hat y^k, \tilde S^k\right), z := (X, y, S)$, $\tilde X^0 = X^0$, $\tilde S^0 = S^0$, and $\hat y^0 = y^0$.
Let the primal-dual gap function be
\begin{equation}
    \label{app:eq: primal-dual gap function}
\begin{aligned}
Q\left(z^k, z\right)= & f\left(\hat y^k\right) +\langle X, \mathcal{A}^*(\hat y^k)+ \tilde S^k-C\rangle 
 -f(y) -\langle \tilde X^k, \mathcal{A}^*(y)+ S-C\rangle .
\end{aligned}
\end{equation}
By~\eqref{app:eq: hat S in QADMM for SDP} and the property of projection operator, we have
$$
\langle \hat S^{k+1} - \tilde V^{k+1}, \hat S^{k+1}-S\rangle \leq 0, \quad \forall S \in \mathbb{S}_+^n.
$$
This together with~\eqref{app:eq: hat X in QADMM for SDP} yields
\begin{equation}
\label{app:eq: proof of ergodic convergence of QADMM for SDP 2}    
\langle \hat X^{k+1} , \hat S^{k+1}-S\rangle \leq 0.
\end{equation}

It follows from~\eqref{app:eq: hat y in QADMM for SDP}--\eqref{app:eq: X in QADMM for SDP} that
\begin{equation}
    \label{app:eq: proof of ergodic convergence of QADMM for SDP 1}
    \begin{aligned}
            b
        &=  \mathcal{A}(X^k + \frac{1}{\gamma}
            (\mathcal{A}^*(\hat y^{k+1}) + S^{k} - C)) \quad \text{(by~\eqref{app:eq: hat y in QADMM for SDP})} \\
        &=  \mathcal{A}(X^k + \frac{1}{\gamma}
        (\mathcal{A}^*(\tilde y^{k+1} - \Delta \tilde y^{k+1}) + S^{k} - C)) \quad \text{(by~\eqref{app:eq: tilde y in QADMM for SDP})} \\
        &= \mathcal{A}( X^{k}) + \mathcal{A}(\hat X^{k+1} - X^{k} - \frac{1}{\gamma} (\hat S^{k+1} - S^{k}))
        \quad \text{(by~\eqref{app:eq: hat V in QADMM for SDP} and~\eqref{app:eq: hat X in QADMM for SDP})} \\
        &=  \mathcal{A}( \hat X^{k+1} - \frac{1}{\gamma} (\hat S^{k+1} - S^{k})).
    \end{aligned}
\end{equation}

Combining this with~\eqref{app:eq: primal-dual gap function} and~\eqref{app:eq: proof of ergodic convergence of QADMM for SDP 2}, we have
$$
\begin{aligned}
 Q\left(z^{k+1}, z\right) 
 & = -b^\top(\hat y^{k+1} - y)  + \langle X, \mathcal{A}^* (\hat y^{k+1})+ \tilde S^{k+1}-C\rangle - \langle \tilde X^{k+1}, \mathcal{A}^* (y )+ S-C\rangle \\
& \leq -b^\top(\hat y^{k+1} - y)  + \langle X, \mathcal{A}^* (\hat y^{k+1})+ \tilde S^{k+1}-C\rangle - \langle \tilde X^{k+1}, \mathcal{A}^* (y )+ S-C\rangle \\
& \quad - \langle \hat X^{k+1} , \hat S^{k+1}-S\rangle \quad \text{(by~\eqref{app:eq: proof of ergodic convergence of QADMM for SDP 2})} \\
& = - \langle \hat X^{k+1} - \frac{1}{\gamma} (\tilde S^{k+1} - S^{k}), \mathcal{A}^*\left(\hat y^{k+1}-y\right)\rangle - \langle \Delta \tilde S^{k+1}, \mathcal{A}^*\left(\hat y^{k+1}-y\right)\rangle\\
&\quad   + \langle X, \mathcal{A}^* (\hat y^{k+1})+ \tilde S^{k+1}-C\rangle - \langle \tilde X^{k+1}, \mathcal{A}^* (y )+ S-C\rangle  - \langle \hat X^{k+1} , \tilde S^{k+1}-S\rangle\\
&\quad + \langle \hat X^{k+1}, \Delta \tilde S^{k+1}\rangle
\quad \text{(by~\eqref{app:eq: proof of ergodic convergence of QADMM for SDP 1} and~\eqref{app:eq: tilde S in QADMM for SDP})} \\
&= - \langle \hat X^{k+1} - \frac{1}{\gamma} (\hat S^{k+1} - S^{k}), \mathcal{A}^*\left(\hat y^{k+1}-y\right)\rangle  + \langle X - \tilde X^{k+1}, \mathcal{A}^* (\hat y^{k+1})+  \tilde S^{k+1}-C\rangle \\
&\quad + \langle\tilde X^{k+1}, \mathcal{A}^* (\hat y^{k+1})+  \tilde S^{k+1}-C\rangle - \langle \tilde X^{k+1}, \mathcal{A}^* (y )+ S-C\rangle - \langle \hat X^{k+1} , \tilde S^{k+1}-S\rangle \\
&\quad + \langle \Delta \tilde S^{k+1}, \hat X^{k+1} - \mathcal{A}^*\left(\hat y^{k+1}-y\right)\rangle \\
&= \langle X - \tilde X^{k+1}, \mathcal{A}^* (\hat y^{k+1})+  \tilde S^{k+1}-C\rangle + \frac{1}{\gamma} \langle \hat S^{k+1}-S^{k}, \mathcal{A}^*\left(\hat y^{k+1}-y\right)\rangle \\
&\quad + \langle \tilde X^{k+1} - \hat X^{k+1} , \mathcal{A}^*(\hat y^{k+1}) + \tilde S^{k+1} - C\rangle 
+ \langle \Delta \tilde S^{k+1}, \hat X^{k+1} - \mathcal{A}^*\left(\hat y^{k+1}-y\right)\rangle \\
&= \gamma \langle X - \tilde X^{k+1}, \hat X^{k+1} - X^k +\frac{1}{\gamma} \Delta V^{k+1} \rangle + \frac{1}{\gamma} \langle \hat S^{k+1}-S^{k}, \mathcal{A}^*\left(\hat y^{k+1}-y\right)\rangle \\
&\quad +\gamma \langle \Delta \tilde X^{k+1} , \hat X^{k+1} - X^k\rangle + \langle \Delta \tilde S^{k+1}, \hat X^{k+1} - \mathcal{A}^*\left(\hat y^{k+1}-y\right)\rangle  \quad \text{(by~\eqref{app:eq: hat X in QADMM for SDP},~\eqref{app:eq: tilde X in QADMM for SDP},~\eqref{app:eq: hat V in QADMM for SDP})} \\
&\leq \gamma \langle X- \tilde X^{k+1}, \tilde X^{k+1}- \tilde X^{k}\rangle+\frac{1}{\gamma}\langle \tilde S^{k+1}-\tilde S^{k}, \mathcal{A}^*\left(\hat y^{k+1}-y\right)\rangle + e^{(1)} ,
\end{aligned}
$$
where $e^{(1)}
= 2\delta_V R_X + \delta_{\tilde X} R_X + \delta_{\tilde S}  (R_X + 2 R_y)
$ is the upper bound of $ \langle X - \tilde X^{k+1}, \Delta V^{k+1}\rangle + \gamma \langle \Delta \tilde X^{k+1} , \hat X^{k+1} - X^k\rangle + \langle \Delta \tilde S^{k+1}, \hat X^{k+1} - \mathcal{A}^*\left(\hat y^{k+1}-y\right)\rangle  $.
Note that
$$
2\langle X-\tilde X^{k+1}, \tilde X^{k+1}-\tilde X^{k}\rangle=\left\|X-\tilde X^{k}\right\|_{\mathrm{F}}^2-\left\|X-\tilde X^{k+1}\right\|_{\mathrm{F}}^2-\left\|\tilde X^{k}-\tilde X^{k+1}\right\|_{\mathrm{F}}^2 .
$$
and
$$
\begin{aligned}
 2\langle \tilde S^{k+1}-\tilde S^{k}, \mathcal{A}^*(\hat y^{k+1}-y)\rangle 
& =\left\|\mathcal{A}^* (y)+ \tilde S^{k}-C\right\|_{\mathrm{F}}^2
-\left\|\mathcal{A}^* (y)+ \tilde S^{k+1}-C\right\|_{\mathrm{F}}^2 \\
&\quad +\left\|\mathcal{A}^* (\hat y^{k+1})+ \tilde S^{k+1}-C\right\|_{\mathrm{F}}^2
-\left\|\mathcal{A}^* (\hat y^{k+1})+ \tilde S^{k}-C\right\|_{\mathrm{F}}^2 \\
& =\left\|\mathcal{A}^* (y)+ \tilde S^{k}-C\right\|_{\mathrm{F}}^2-\left\|\mathcal{A}^* (y)+ \tilde S^{k+1}-C\right\|_{\mathrm{F}}^2\\
&\quad +\gamma^2\left\|X^{k}- \hat X^{k+1} - \frac{1}{\gamma}\Delta V^{k+1}\right\|_{\mathrm{F}}^2-\left\|\mathcal{A}^* (\hat y^{k+1})+ \tilde S^{k}-C\right\|_{\mathrm{F}}^2 \\
& \leq \left\|\mathcal{A}^* (y)+ \tilde S^{k}-C\right\|_{\mathrm{F}}^2-\left\|\mathcal{A}^* (y)+ \tilde S^{k+1}-C\right\|_{\mathrm{F}}^2\\
&\quad +\gamma^2\left\|\tilde X^{k}- \tilde X^{k+1}\right\|_{\mathrm{F}}^2-\left\|\mathcal{A}^* (\hat y^{k+1})+ \tilde S^{k}-C\right\|_{\mathrm{F}}^2 + e^{(2)} ,
\end{aligned}
$$
where $e^{(2)} = \gamma^2 (\delta_X + \delta_{\tilde X} + \frac{1}{\gamma}\delta_V) (4R_X + \frac{1}{\gamma}\delta_V)$ is the upper bound of $ \gamma^2\left\| X^{k}- \hat X^{k+1} - \frac{1}{\gamma}\Delta V^{k+1}\right\|_{\mathrm{F}}^2-\gamma^2\left\|\tilde X^{k}- \tilde X^{k+1}\right\|_{\mathrm{F}}^2 $.
Thus we conclude that
$$
\begin{aligned}
&\quad Q\left(z^{k+1}, z\right)  \\
&\leq \frac{\gamma }{2}\left(\left\|\tilde X^{k}-X\right\|_{\mathrm{F}}^2-\left\|\tilde X^{k+1} -  X\right\|_{\mathrm{F}}^2\right) \\
& +\frac{1}{2\gamma}\left(\left\|\mathcal{A}^* y+ \tilde S^{k}-C\right\|_{\mathrm{F}}^2-\left\|\mathcal{A}^* (y)+ \tilde S^{k+1}-C\right\|_{\mathrm{F}}^2-\left\|\mathcal{A}^* (\hat y^{k+1})+ \tilde S^{k}-C\right\|_{\mathrm{F}}^2\right) + e^{(1)} + e^{(2)} \\
&\leq \frac{\gamma }{2}\left(\left\|\tilde X^{k}-X\right\|_{\mathrm{F}}^2-\left\|\tilde X^{k+1}-X\right\|_{\mathrm{F}}^2\right) \\
& +\frac{1}{2\gamma}\left(\left\|\mathcal{A}^* (y)+ \tilde S^{k}-C\right\|_{\mathrm{F}}^2-\left\|\mathcal{A}^* (y)+ \tilde S^{k+1}-C\right\|_{\mathrm{F}}^2\right) + e^{(1)} + e^{(2)} .
\end{aligned}
$$
Summing up the above inequality from $k=0, \ldots, K-1$, we obtain
\begin{equation}
    \label{app:eq: sum of primal-dual gap function}
\begin{aligned}
\sum_{k=1}^{K} Q\left(z^k, z\right) \leq & \frac{\gamma}{2}\left(\left\|\tilde X^0-X\right\|_{\mathrm{F}}^2-\left\|\tilde X^K-X\right\|_{\mathrm{F}}^2\right)  \\
& +\frac{1}{2\gamma}\left(\left\|\mathcal{A}^*(y) + \tilde S^{0} - C\right\|_{\mathrm{F}}^2-\left\|\mathcal{A}^*(y) + \tilde S^{K} - C\right\|_{\mathrm{F}}^2\right) + K e^{(1)} + K e^{(2)} .
\end{aligned}
\end{equation}
Setting $z=z^*$ in the above inequality and using the facts that 
$\mathcal{A}^*y^* + \tilde S^k - C = \tilde S^k - S^*$ and
$Q\left({z}^k, z^*\right) \geq 0$, we see that
$$
\left\|\tilde X^K -X^*\right\|_{\mathrm{F}}^2 \leq\left\|\tilde X^0-X^*\right\|_{\mathrm{F}}^2+\frac{1}{\gamma^2}\left\|\tilde S^0-S^*\right\|_{\mathrm{F}}^2 + \frac{2K}{\gamma} (e^{(1)} + e^{(2)} ) ,
$$
and hence that
\begin{equation}
    \label{app:eq: Xk - X0}
    \begin{aligned}
\left\|\tilde X^K-\tilde X^0\right\|_{\mathrm{F}} &\leq\left\|\tilde X^0-X^*\right\|_{\mathrm{F}}+\left\|\tilde X^K-X^*\right\|_{\mathrm{F}} \\
&\leq 2\left\|\tilde X^0-X^*\right\|_{\mathrm{F}}+\frac{1}{\gamma}\left\|\tilde S^0-S^*\right\|_{\mathrm{F}} + \sqrt{\frac{2K}{\gamma}( e^{(1)}+e^{(2)}) }.
\end{aligned}
\end{equation}
From the definition of $(\bar y^K, \bar S^K)$ and~\eqref{app:eq: hat X in QADMM for SDP},~\eqref{app:eq: tilde X in QADMM for SDP},~\eqref{app:eq: hat V in QADMM for SDP}, we have
\begin{equation}
    \label{app:eq: violation of average iterate}
    \begin{aligned}
     \left\| A^* (\bar y^K)  + \bar S^K - C \right\|_{\mathrm{F}} &= \left\| \frac{\gamma}{K} \sum_{k=0}^{K-1} (\tilde X^k - \tilde X^{k-1} - \Delta \tilde X^{k} - \Delta X^{k-1} + \frac{1}{\gamma} \Delta V^{k}) \right\|_{\mathrm{F}} \\
    &\leq \left\| \frac{\gamma}{K} ( \tilde X^K - \tilde X^0 ) \right\|_{\mathrm{F}} + \delta_{\tilde X} + \delta_X + \frac{1}{\gamma} \delta_{\tilde V} \\
    &\leq \frac{\gamma}{K} \left(2\left\|\tilde X^0-X^*\right\|_{\mathrm{F}}+\left\|\tilde X^K-X^*\right\|_{\mathrm{F}} + \sqrt{\frac{2K}{\gamma}( e^{(1)}+e^{(2)}) }\right) \\
    &\quad + \delta_{\tilde X} + \delta_X + \frac{1}{\gamma} \delta_V. \\
    \end{aligned}
\end{equation}
This relation implies that~\eqref{app:eq: ergodic convergence of QADMM for SDP 2} holds.
Moreover, letting $z=\left(X^*, y^*, S^*\right)$ in~\eqref{app:eq: sum of primal-dual gap function} and using the fact that 
$$
\begin{aligned}
\frac{1}{K} \sum_{k=1}^K Q\left(z^K,\left(X^*, y^*, S^*\right)\right) & \geq Q\left(\bar{z}^K,\left(X^*, y^*, S^*\right)\right) \\
& =f\left(\bar{y}^K\right)-f\left(y^*\right)+\langle X^*, \mathcal{A}^* (\bar{y}^k)+ \bar{S}^K-C\rangle ,
\end{aligned}
$$
we have
$$
\begin{aligned}
 f\left(\bar{y}^K\right)-f\left(y^*\right)
& \leq \frac{1}{2 K}\left(\gamma\left\|\tilde X^0 - X^*\right\|_{\mathrm{F}}^2-\gamma\left\|\tilde X^K - X^*\right\|_{\mathrm{F}}^2+\frac{1}{\gamma}\left\|\tilde S^0-S^*\right\|_{\mathrm{F}}^2-\frac{1}{\gamma}\left\| \tilde S^K-S^*\right\|_{\mathrm{F}}^2\right)\\
&\quad  + e^{(1)} + e^{(2)} + \|X^*\|_{\mathrm{F}} \left\| \mathcal{A}^* (\bar{y}^k)+ \bar{S}^K-C\right\|_{\mathrm{F}} .
\end{aligned}
$$
Combining this with~\eqref{app:eq: violation of average iterate} gives~\eqref{app:eq: ergodic convergence of QADMM for SDP 1}.
This completes the proof.

\end{proof}

\begin{corollary}
    \label{app:cor: ergodic convergence of QADMM for SDP 2}
    Let the initial point be $(X^0, y^0, S^0) = (0, 0, 0)$. Then the output of QADMM satisfies
    \begin{equation}
    \begin{aligned}
    f({y}_{out}) - f(y^*) &
    \leq \mathcal{O}\left(\frac{R_X^2+R_S^2}{K} + \delta \right) ,
    \\
     \left\|\mathcal{A}^*( y_{out} ) + S_{out}-C\right\|_2 & \leq
     \mathcal{O}\left(\frac{R_X+R_S}{K} + \delta  \right) .
    \end{aligned} 
    \end{equation}
\end{corollary}
\begin{proof}
    Since $\| y_{out} - \bar y^K \|_2 \leq \delta_{\tilde y} + \delta_y$ and $\| S_{out} - \bar S^K \|_2 \leq \delta_S$, the desired result follows from Theorem~\ref{app:thm: ergodic convergence of QADMM for SDP}.
\end{proof}

\begin{corollary}
    Assume $ \delta = \Theta(\epsilon_{\mathrm{abs}})$ and let the initial point be $(X^0, y^0, S^0) = (0, 0, 0)$.
    To achieve an $\epsilon_{\mathrm{abs}}$-accuracy solution $(y_{out}, S_{out})$ satisfying
    $
    - b^\top y_{out} + b^\top y^* \leq \epsilon_{\mathrm{abs}}
    $ and $\|A^*(y_{out}) + S_{out} - C\|_2 \leq \epsilon_{\mathrm{abs}}$,
    QADMM requires 
    $$
    K = \mathcal{O}\left( \frac{R_X^2 + R_S^2}{\epsilon_{\mathrm{abs}}} \right) 
    $$
    iterations, where $R_X$ and $R_S$ are the upper bounds of $\| X^*\|_{\mathrm{F}}$ and $\| S^* \|_{\mathrm{F}}$, respectively.
\end{corollary}
\begin{proof}
    This follows from Corollary~\ref{app:cor: ergodic convergence of QADMM for SDP 2} and the fact that $\delta = \Theta(\epsilon_{\mathrm{abs}})$.
\end{proof}

\subsection{Complexity of QADMM for SDP}

\begin{theorem}
    \label{app:thm: complexity of QADMM for SDP}
    With the accuracy choice ${\delta_{\hat y}, \delta_{\tilde y}, \delta_y, \delta_V, \delta_S, \delta_X, \delta_{\tilde S}, \delta_{\tilde X}} = \Theta(\epsilon_{\mathrm{abs}})$ and $d = \mathcal{O}(\frac{1+R_X+R_y}{\epsilon_{\mathrm{abs}}})$,
    a quantum implementation of QADMM under the polynomial-size QRAM data-access model described in the main text produces an $\epsilon_{\mathrm{abs}}$-accuracy solution using at most $\mathcal{O}\left( (s_A + n^2)  \frac{R_X^2 + R_S^2}{\epsilon_{\mathrm{abs}}} \right)$ classical operations and
    $$
    \widetilde{\mathcal{O}}\left(\left(m \kappa_A^2( 1 + R_y)^2 + n^2 ({\kappa_A^2}( 1 + R_y) + R_X)\right)\frac{(R_X+R_S)^3}{\epsilon_{\mathrm{abs}}^3}\right)
    $$
    quantum gates and queries.
\end{theorem}
\begin{proof}
    The result follows from the complexity of each iteration of QADMM for SDP and the number of iterations in Corollary~\ref{app:cor: ergodic convergence of QADMM for SDP 2}.
\end{proof}

\bibliographystyle{quantum}
\bibliography{ref}

\end{document}

%% file: _macros_quantum.tex
\usepackage[utf8]{inputenc}
\usepackage[english]{babel}
\usepackage[T1]{fontenc}
\usepackage[numbers,sort&compress]{natbib}
\usepackage{amsmath,amsthm,amssymb,amsfonts,amscd}
\usepackage{dsfont,mathrsfs,mathtools,nicefrac,bm}
\usepackage{algorithmic}
\usepackage[ruled,vlined]{algorithm2e}

\usepackage{accents}
\usepackage{comment,url,graphicx,relsize}
\usepackage{dcolumn}
\usepackage{xcolor}
\usepackage{bbm}
\usepackage{diagbox}
\usepackage{tikz}
\usetikzlibrary{calc,patterns,angles,quotes}
\usepackage{makecell}
\usepackage{multirow}
\usepackage{booktabs}
\usepackage{hyperref}

\renewenvironment{proof}{\noindent\textbf{Proof.}\hspace*{.3em}}{\qed\\}
\newenvironment{proof-sketch}{\noindent\textbf{Proof Sketch}
  \hspace*{0.em}}{\qed\bigskip\\}
\newenvironment{proof-idea}{\noindent\textbf{Proof Idea}
  \hspace*{0.em}}{\qed\bigskip\\}
\newenvironment{proof-of-lemma}[1][{}]{\noindent\textbf{Proof of Lemma {#1}.}
  \hspace*{0.em}}{\qed\\}
\newenvironment{proof-of-corollary}[1][{}]{\noindent\textbf{Proof of Corollary {#1}.}
  \hspace*{0.em}}{\qed\\}
\newenvironment{proof-of-theorem}[1][{}]{\noindent\textbf{Proof of Theorem {#1}.}
  \hspace*{0.em}}{\qed\\}
\newenvironment{proof-attempt}{\noindent\textbf{Proof Attempt}
  \hspace*{0.em}}{\qed\bigskip\\}

\newtheorem{theorem}{Theorem}[section]
\newtheorem{lemma}{Lemma}[section]
\newtheorem{corollary}{Corollary}[section]
\newtheorem{proposition}{Proposition}[section]

\newtheorem{definition}{Definition}[section]

\allowdisplaybreaks
\numberwithin{equation}{section}

\newcommand*{\colorboxed}{}
\def\colorboxed#1#{%
  \colorboxedAux{#1}%
}
\newcommand*{\colorboxedAux}[3]{%
  \begingroup
    \colorlet{cb@saved}{.}%
    \color#1{#2}%
    \boxed{%
      \color{cb@saved}%
      #3%
    }%
  \endgroup
}